\documentclass[12 pt]{article}
\usepackage{amsmath, amsthm, amssymb,enumerate}
\usepackage{tikz}
\usepackage{fullpage}
\usepackage{enumitem}
\usepackage{hyperref}
\usepackage{xcolor}

\title{Maximising the Number of Cycles in Graphs with Forbidden Subgraphs}
\author{Natasha Morrison\thanks{Department of Pure Mathematics and Mathematical Statistics, University of
Cambridge, Wilberforce Road, Cambridge, United Kingdom
 and Instituto Nacional de Matem\'{a}tica Pura e Aplicada, Rio de Janeiro, RJ, Brasil. Research partially supported by CNPq and Sidney Sussex College, Cambridge.
\newline E-mail: \texttt{morrison@dpmms.cam.ac.uk}.}  \and Alexander Roberts\thanks{Mathematical Institute, University of Oxford, Woodstock Road, Oxford, United Kingdom. \newline  E-mail: \texttt{\{robertsa, scott\}@maths.ox.ac.uk}.} \and Alex Scott\footnotemark[2]\footnotemark[2] \thanks{Supported by a Leverhulme Trust Research Fellowship.}}

\newtheoremstyle{case}{}{}{\normalfont}{}{\itshape}{:}{ }{}

\newtheorem{thm}{Theorem}[section]
\newtheorem{lem}[thm]{Lemma}
\newtheorem{claim}[thm]{Claim}

\newtheorem{conj}[thm]{Conjecture}

\theoremstyle{definition}

\newtheorem*{ack}{Acknowledgements}

\newtheorem{qu}[thm]{Question}

\numberwithin{equation}{section}

\newtheoremstyle{case}{}{}{\normalfont}{}{\itshape}{\normalfont:}{ }{}

\theoremstyle{case}

\def\comment#1{}

\newcommand{\ex}{{\rm ex}}
\newcommand{\Ex}{{\rm EX}}
\newcommand{\beq}{\begin{eqnarray*}}
\newcommand{\eeq}{\end{eqnarray*}}

\def\build#1_#2^#3{\mathrel{\mathop{\kern 0pt#1}\limits_{#2}^{#3}}}
\newcommand{\beqs}{\begin{eqnarray}}
\newcommand{\eeqs}{\end{eqnarray}}

\newcommand{\bP}{\mathbb{P}}

\newcommand{\bN}{\mathbb{N}}
\newcommand{\bR}{\mathbb{R}}

\newcommand{\bE}{\mathbb{E}}

\newcommand{\cA}{\mathcal{A}}

\numberwithin{equation}{section}

\begin{document}

\maketitle

\begin{abstract}
Fix $k \ge 2$ and let $H$  be a graph with $\chi(H) = k+1$ containing a critical edge. We show that for sufficiently large $n,$ the unique $n$-vertex $H$-free graph containing the maximum number of cycles is $T_k(n)$. This resolves both a question and a conjecture of Arman, Gunderson and Tsaturian \cite{Gund1}.
\end{abstract}

\section{Introduction}

For a graph $G$, let $c(G)$ be the number of cycles in $G$.  The problem of bounding $c(G)$ for various classes of graph has a long history: for example, an upper bound on $c(G)$ in terms of the cyclomatic number of $G$ was given by Ahrens \cite{ahrens} in 1897; while a lower bound is implicit in work of Kirchhoff \cite{kirchhoff} from fifty years earlier.   

For graphs on $n$ vertices, the number of cycles is clearly maximized by the complete graph, which has $\sum_{i=3}^n (i!/2i)\binom ni$ cycles.  
But what happens if we constrain the structure of $G$ by forbidding some subgraph?  In other words, what is the maximal number of cycles in an $H$-free graph on $n$ vertices  (here a graph is {\em $H$-free} if it does not contain a subgraph isomorphic to $H$)?  For graphs $G$ and $H$, let $c(G)$ be the number of cycles in $G$ and let 
$$m(n;H):= \max\{ c(G): |V(G)| = n, H \not\subseteq G\}.$$
The problem of determining $m(n,H)$  was introduced by Durocher,   Gunderson,  Li and Skala \cite{Gund2} (who studied $m(n,K_3)$)  and will be the focus of this paper.

The problem of maximizing the number of {\em edges} in an $H$-free graph has been extensively studied.   Indeed, Tur\'{a}n \cite{turan} proved that the unique $n$-vertex $K_{k+1}$-free graph with the maximum number of edges is the complete $k$-partite graph with all classes of size $\lfloor n/k\rfloor$ or $\lceil n/k \rceil$, which is known as the \emph{Tur\'{a}n graph} $T_k(n)$. More generally, the classical Tur\'{a}n problem asks for the maximum number of edges in an $H$-free graph:  this is the \emph{extremal number} $\ex(n;H)$
and the {\em extremal graphs} are $\Ex(n;H) = \left\{G : \left|V(G)\right| = n, H \not\subseteq G\right\}$, that is the $H$-free graphs on $n$ vertices with $\ex(n;H)$ edges. 
For further detail, we refer to \cite{boletg}. 

Much less is known about maximizing the number of {\em cycles} in $H$-free graphs.
Durocher, Gunderson, Li and Skala \cite{Gund2} investigated $m(n,K_3)$, and conjectured that the maximum is attained by the Tur\'an graph $T_2(n)$.
This conjecture was proved for large $n$ by Arman, Gunderson and Tsaturian \cite{Gund1}, who showed that, for $n \ge 141$, $T_2(n)$ is the unique triangle-free graph containing $m(n;K_3)$ cycles. They made the following natural further conjecture.

\begin{conj}[Arman, Gunderson and Tsaturian \cite{Gund1}]\label{gundconj}
For any $k > 1$, for sufficiently large $n$, $T_2(n)$ is the unique $n$-vertex $C_{2k+1}$-free graph containing $m(n;C_{2k+1})$ cycles.
\end{conj}

A partial result towards this conjecture is given in \cite{Gund1}, where it is shown that $m(n;C_{2k+1}) = O(c(T_2(n)))$. They also ask about a different generalisation.

\begin{qu}[Arman, Gunderson and Tsaturian \cite{Gund1}]\label{gundqu}
For $k \ge 4$, what is $m(n;K_k)$? Is $T_{k-1}(n)$ the $K_k$-free graph containing $m(n;K_k)$ cycles?
\end{qu}

In this paper we prove Conjecture~\ref{gundconj} for any fixed $k$ and sufficiently large $n$ and answer Question~\ref{gundqu} affirmatively for sufficiently large $n$. In fact we prove a much more general result. In what follows we say that an edge $e$ of a graph $H$ is \emph{critical} if  $\chi(H \backslash \{e\}) = \chi(H) - 1$. Our main result is the following.

\begin{thm}\label{main}
Let $k \ge 2$ and let $H$ be a graph with $\chi(H) = k+1$ containing a critical edge. Then for sufficiently large $n$, the unique $n$-vertex $H$-free graph containing the maximum number of cycles is the Tur\'{a}n graph $T_k(n)$.
\end{thm}

The condition that $H$ has a critical edge is necessary, since if $H$ does not have a critical edge we can add an edge to the relevant Tur\'{a}n graph without creating a copy of $H$ (and the addition of this edge will increase the number of cycles). Conjecture~\ref{gundconj} follows from Theorem~\ref{main} as an odd cycle contains a critical edge. 

By using the same techniques as in the proof of Theorem~\ref{main}, we are able to obtain a bound on the number of cycles in an $H$-free graph for any fixed graph $H$ (not just critical ones).

\begin{thm}\label{cyclecount}
Let $k \ge 2$ and l\textit{•}et $H$ be a fixed graph with $\chi(H) =  k+1$. Then
$$ m(n;H) \le \left(\frac{k-1}{k}\right)^nn^ne^{-(1-o(1))n}.$$
\end{thm}

The Tur\'{a}n graph gives a lower bound showing that this bound is tight up to the $o(1)$ term in the exponent.

In this paper we concern ourselves with maximising cycles of any length in a graph with a forbidden subgraph. The related problem of maximising copies of a single graph in a graph with a collection of forbidden subgraphs has received a great deal of attention. For a graph $G$ and family of graphs $\mathcal{F}$, define $\ex(n,G,\mathcal{F})$ to be the maximum possible number of copies of $G$ in a graph containing no member of $\mathcal{F}$.
The value of $\ex(n,G,\mathcal{F})$ is of particular interest when the graphs being studied are cycles (see \cite{NACS,BB,ER} for results concerning other graphs). 
Improving on earlier work of Bollob\'{a}s and Gy\H{o}ri \cite{BG} and Gy\H{o}ri and Li \cite{GL}, Alon and Shikhelman \cite{NACS} gave bounds for $\ex(n,K_3,C_{2k+1})$, when $k \ge 2$. Using flag algebras, Hatami, Hladk\'{y}, Kr\'al', Norine, and Razborov \cite{HHKNR} showed that the unique triangle-free graph with maximum number of copies of $C_5$ is the balanced blow up of $C_5$. Also using flag algebras, Grzesik \cite{GG} determined $\ex(n,C_5,K_3)$. More recently, Grzesik and Kielak \cite{AGBK} determined $\ex(n,C_{2k+1},\mathcal{F})$, where $k \ge 3$ and $\mathcal{F}$ is the family of odd cycles of length at most $2k-1$. They also asymptotically determine $\ex(n,C_{2k+1},C_{2k-1})$.

The rest of paper is organised as follows. Section~\ref{seckpart} contains a number of lemmas about counting cycles in complete $k$-partite graphs (Lemmas~\ref{kKMain}-\ref{secondcount}). These will be used in Section~\ref{secmain} for the proof of Theorem~\ref{main}. The statements are very natural but our proofs are unfortunately technical, so we defer these to Section~\ref{sectech}. In Section \ref{sec2} we prove Lemma~\ref{easycor} and use similar techniques to prove Theorem \ref{cyclecount}. The proof of Theorem~\ref{main} is completed in Section \ref{secmain}. We conclude the paper in Section \ref{sec5} with some related problems and open questions.  We conclude the current section with a sketch of the proof of Theorem \ref{main}.

\subsection{Outline of Proof}

In what follows we fix $H$ to be a graph with $\chi(H) = k+1$ that contains a critical edge and assume that $n$ is sufficiently large. As usual, for a graph $F$ we will write $e(F):= |E(F)|$ and in the particular case of the Tur\'{a}n graph, we will write $t_k(n):= |E(T_k(n))|$. Let $G$ be an $n$-vertex $H$-free graph with $c(G) = m(n;H)$. As $T_k(n)$ is $H$-free, we have that $m(n;H) \ge c(T_k(n))$. We will suppose that $G$ is not $T_k(n)$ and obtain a contradiction by showing that $c(G) < c(T_k(n))$.

 The first step in the proof (Lemma~\ref{edgecount}) is to show that $G$ with $c(G) \ge c(T_k(n))$ contains at least $e(T_k(n)) - O(n \log^2 n)$ edges. In order to prove this, we will need a bound on the number of cycles an $n$-vertex $H$-free graph with $m \ge \beta(H)\cdot n$ edges can contain, where $\beta$ is some constant depending on $H$. Such a bound is provided by Lemma~\ref{easycor}. 

Given Lemma~\ref{edgecount}, we are able to apply the following stability result from~\cite{stabAA}.

\begin{thm}[Theorem 1.4 \cite{stabAA}]\label{stable}
Let $H$ be a graph with a critical edge and $\chi(H) = k+1 \ge 3$, and let $f(n) = o(n^2)$ be a function.
If $G$ is an $H$-free graph with $n$ vertices and $e(G) \ge t_k(n) - f(n)$ then $G$ can be made $k$-partite by deleting $O(n^{-1}f(n)^{3/2})$ edges.
\end{thm}

Since we have $f(n) = O(n\log^2 n)$, this will imply that $G$ is a sublinear number of edges away from being $k$-partite. We then take a $k$-partition of $G$ which minimises the number of edges within classes and carefully bound (given that $G$ is not $T_k(n)$) the number of cycles $G$ can contain that do not use edges within classes (Lemma~\ref{regcycle}). We conclude the proof by separately counting the cycles in $G$ that use edges within classes and observing that the total number of cycles in $G$ is not large enough, a contradiction.

\section{Counting Cycles in Complete $k$-partite Graphs}\label{seckpart}
In this section we state some results about the number of cycles in complete $k$-partite graphs. These are needed in Section~\ref{secmain} for the proof of Theorem~\ref{main}, but may be of independent interest. Despite the simplicity of the statements, the proofs are annoyingly technical, and so we will give them later in Section~\ref{sectech}.

The first gives a bound on the number of cycles in $T_k(n)$. In what follows we write $h(G)$ for the number of Hamiltonian cycles in $G$ (a Hamiltonian cycle of a graph is a cycle covering all of the vertices). We also define $c_r(G)$ to be the number of cycles of length $r$ in $G$. 
\begin{lem}\label{kKMain}
$$c_{2\lfloor n/2 \rfloor}\left(T_2(n)\right) \sim \pi2^{-n}n^ne^{-n},$$
and for fixed $k \ge 3$, 
$$h(T_k(n)) = \Omega\left(\left(\frac{k-1}{k}\right)^nn^{n-\frac{1}{2}}e^{-n}\right).$$

\end{lem}

Since $c(G) \ge h(G)$ for all $G$, if follows that $c(T_k(n)) = \Omega\left(\left(\frac{k-1}{k}\right)^nn^{n-\frac{1}{2}}e^{-n}\right).$ Arman \cite[Theorems 5.22 and 5.26]{arman} proves similar results here and also provides an upper bound for $c(T_k(n))$.

\begin{lem}\label{Turanbest}
Let $k \ge 2$ and G be an $n$-vertex $k$-partite graph. Then for any $r$, $c_r(T_k(n)) \ge c_r(G)$. Furthermore, when $n \ge 5,$ $c(T_k(n)) > c(G)$ for any $n$-vertex $k$-partite graph $G$ not isomorphic to $T_k(n)$.
\end{lem}

In particular, Lemma \ref{Turanbest} implies that the Tur\'{a}n graph $T_k(n)$ has the most Hamilton cycles amongst all $k$-partite graphs on $n$ vertices.

In order to state the next few lemmas we require some more technical definitions. For $\underline{a} = (a_1,\ldots,a_k) \in \mathbb{N}^k$, we define $K_{\underline{a}}$ to be the complete $k$-partite graph with vertex classes $V_1,\ldots, V_k$, where $|V_i| = a_i$. Let $v$ be some vertex in $V(K_{\underline{a}})$. We define $h_v(j,K_{\underline{a}})$ to be the number of permutations $v_1 \cdots v_n$ of the vertices of $K_{\underline{a}}$, such that $v_1 = v$, $v_2 \in V_j$ and $v_1 \cdots v_n$ is a Hamilton cycle (we count permutations rather than cycles, so that we count a cycle $v_1\cdots v_n$ with $v_2$ and $v_n$ from the same vertex class twice). Note that if we count the Hamilton cycles by considering $v_1\cdots v_n$ with $v_1$ fixed, by counting the number of cycles visiting each other vertex class first, then each cycle will be counted twice due to the choice of orientation. So for $v \in V_i$, we have
	\begin{align}
		h(K_{\underline{a}}) = \frac{1}{2}\sum_{j \neq i}h_v(j,K_{\underline{a}}).\label{first1}
	\end{align}

The next lemma will allow us to count cycles more accurately in complete $k$-partite graphs that are not balanced.
\begin{lem}\label{close}
Let $k \ge 3$. Let $\underline{b} = (b_1,\ldots,b_k)$, $\underline{c} = (c_1,\ldots,c_k) \in \bN^k$ be such that $b_i \ge b_j$ if and only if $c_i \ge c_j$, and that $K_{\underline{b}} \cong T_k(n)$. Denote the vertex classes of $K_{\underline{c}}$ by $V_1,\ldots,V_k$, and vertex classes of $K_{\underline{b}}$ by $V_1',\ldots, V_k'$. Then if $v \in V_1, w \in V_1'$, then
	\begin{align}
		h_v(2,K_{\underline{c}}) \le h_w(2,T_k(n))\prod_{i =1}^ke^{\left|\log(\frac{b_i}{c_i})\right|}. \nonumber
	\end{align}
\end{lem}
We now bound the proportion of Hamilton cycles starting from a fixed vertex that immediately pass through a fixed vertex class. This will be important when we bound the cycles in a non-complete $k$-partite graph.

\begin{lem}\label{Turancount}
Let $k \ge 3$, and suppose $T_k(n)$ has vertex classes $V_1,\ldots,V_k$ (arbitrarily ordered independently of class size). Then for $n$ sufficiently large, if $v \in V_1$,
	\begin{align}
		h_v(2,T_k(n)) \ge \frac{2}{3k}h(T_k(n)). \nonumber
	\end{align}
\end{lem}

The next two lemmas give a recursive bound on the number of Hamilton cycles in $T_k(n)$. This will allow us to bound the number of cycles in the Tur\'{a}n graph in terms of the number of Hamilton cycles it contains. Throughout the chapter we will make use of the notation $(n)_i := n \cdot (n-1) \cdots (n - (i-1))$.
\begin{lem}\label{recursion}
For $k,n \in \bN, k \ge 3$ and $i \in [n]$,
	\begin{align}
		h(T_k(n)) \ge (n-1)_i\left(\frac{k-2}{k}\right)^ih(T_k(n-i)). \nonumber
	\end{align}
\end{lem}

\begin{lem}\label{secondcount}
For $k,n \in \bN, k \ge 3$:
	\begin{align}
		c(T_k(n)) \le e^{\frac{2k}{k-2}}h(T_k(n)). \nonumber
	\end{align}
\end{lem}

Finally, we have similar results when $k = 2$. This case is slightly different to when $k\ge 3$ as $T_2(n)$ only contains even cycles.

\begin{lem}\label{second2count}
For $n \in \mathbb{N}$ and $i = o(n)$, we have

	$$c(T_2(n-i)) \le 2e \left(\frac{4}{n}\right)^i c_{2\left\lfloor \frac{n}{2} \right\rfloor}(T_2(n)).$$

\end{lem}

\section{Counting Cycles in $H$-free Graphs}\label{sec2}
Fix $H$ to be a graph with $\chi(H) =k+1$ containing a critical edge. 
The first aim of this section is to prove a lemma bounding the number of cycles in an $n$-vertex $H$-free graph containing a fixed number of edges. We will need the following theorem of Simonovits~\cite{Sim1}.

\begin{thm}[Simonovits {\cite[Theorem 2.3]{Sim1}}]
\label{simthm}
Let $H$ be a graph with $\chi(H) =k+1 \ge 3$ that contains a critical edge. Then there exists some $n_0$ such that, for all $n \ge n_0$, we have $\Ex(n;H) = \{T_k(n)\}$.
\end{thm}

Given $H$, define $n'_0(H)$ to be the smallest value of $n_0$ such that Theorem~\ref{simthm} holds and choose $n_0(H) \ge n'_0(H)$ such that $\ex(n;H) \ge 10n$ for each $n \ge n_0$. We define $\beta(H):= 10n_0.$

In a recent paper, Arman and Tsaturian~\cite{armtsat} consider the maximum number of cycles in a graph with a fixed number of edges: They show that if $G$ is an $n$-vertex graph with $m$ edges, then 
$$c(G)  \le \left\{ \begin{array}{c l}      
    \frac{3}{4}\Delta(G)\left(\frac{m}{n-1}\right)^{n-1} & \text{ for } \frac{m}{n-1} \ge 3,\\
     \frac{3}{4}\Delta(G)\cdot \left(\sqrt[\leftroot{-2}\uproot{2}3]{3}\right)^m, & \text{otherwise}.\\
\end{array}\right.$$

This general bound is not strong enough for us: comparing this bound with the bounds given in Lemma \ref{kKMain}, we see that a graph with at least as many cycles as $T_k(n)$ has at least $\left(1+o(1)\right)e^{-1}t_k(n)$ edges. However under the additional assumption that our graph does not contain a forbidden subgraph $H$, we are able to prove the following lemma which we will later use to show that an $H$-free graph with at least as many cycles as $T_k(n)$ has at least $\left(1+o(1)\right)t_k(n)$ edges. We remark that when $m$ is close to $t_k(n),$ the bound we gives beats the general bound of Arman and Tsaturian by an exponential factor.

\begin{lem}\label{easycor}
Let $H$ be a fixed graph with $\chi(H) =k+1 \ge 3$ containing a critical edge. For $n$ sufficiently large, let $G$ be an $H$-free graph with $n$ vertices and $m$ edges where $t_k(n) - 10n \ge m \ge \beta(H) \cdot n$ (recall the definition of $\beta(H)$ from just after Theorem \ref{simthm}). Then $c(G) = O\left(\lambda^{n}n^{n+2}\left(\frac{k-1}{k}\right)^ne^{\frac{2k-1}{(k-1)\lambda} - \lambda n}\right)$, where 

	\begin{align}
		\lambda := 1 - \left(1 - \frac{2k}{k-1}\frac{m}{\left(n-3\right)^2}\right)^{\frac{1}{2}}.\label{alphadef}
	\end{align}
\end{lem}

The next lemma bounds the maximum number of paths that an $H$-free graph $G$ can contain between two fixed vertices. For $x,y \in V(G)$, define $p_{x,y}$ to be the number of paths between $x$ and $y$ in $G$.

\begin{lem}\label{count}
Let $H$ be a graph with $\chi(H) = k+1 \ge 3$ that contains a critical edge. For $n$ sufficiently large, let $G$ be an $H$-free graph with $n$ vertices and $m$ edges where $t_k(n) - 10n \ge m \ge \beta(H) \cdot n$ (recall the definition of $\beta(H)$ from just after Theorem \ref{simthm}). Then for any $x,y \in V(G)$,
$$p_{x,y}(G) = O\left(\lambda^{n}n^{n}\left(\frac{k-1}{k}\right)^ne^{\frac{2k-1}{(k-1)\lambda} - \lambda n}\right),$$
where $\lambda$ is as defined in \eqref{alphadef}. 
\end{lem}

Lemma~\ref{easycor} follows easily from Lemma~\ref{count}.

\begin{proof}[Proof of Lemma~\ref{easycor}]
Observe that for each edge $e = xy$ in $G$, the number of cycles containing $e$ is at most $p_{x,y}$. Thus, by Lemma \ref{count}
\begin{align*}
c(G) &\le \sum_{xy \in E(G)}p_{x,y}(G) \\
&= O\left(m\lambda^{n}n^{n}\left(\frac{k-1}{k}\right)^ne^{\frac{2k-1}{(k-1)\lambda} - \lambda n}\right)\\
&= O\left(\lambda^{n}n^{n+2}\left(\frac{k-1}{k}\right)^ne^{\frac{2k-1}{(k-1)\lambda} - \lambda n}\right),
\end{align*}
as required.
\end{proof}

Before proving Lemma~\ref{count}, we prove the following Lemma which allows us to consider an integer valued linear optimisation problem to find upper bounds for the number of paths between vertices in graphs with a forbidden subgraph.

\begin{lem}\label{ref3count}
Let $H$ be a graph with $\chi(H) \ge 3.$ Let $G$ be an $H$-free graph with $n$ vertices and $m$ edges, and let $x,y$ be vertices of $G.$ Then $p_{x,y}(G)$ is bounded by the maximum value of the product
\begin{equation}
\label{prod}
\prod_{i=2}^{n} \max\{r_i,1\}
\end{equation}
under the following set of constraints:
\begin{enumerate}
\item[(i)] $r_i \in \mathbb{Z}_{\ge 0}$, for $2 \le i \le n,$
\item[(ii)] $\sum_{i=2}^n r_i \le m,$ and
\item[(iii)] $\sum_{i=2}^t r_i \le \ex(t;H)$, for $2 \le t \le n$.
\end{enumerate}
\end{lem}

\begin{proof}[Proof of Lemma \ref{ref3count}]
Fix $x,y \in V(G)$. We define a sequence of vertices $(x_i)_{i\in [n]}$ and a sequence of graphs $(G_i)_{i\in[n]}$ as follows. Let $x_1 = x$ and $G_1 = G$. For $i \ge 2$, given $x_{i-1}$ and $G_{i-1}$, let $G_i = G_{i-1}\setminus x_{i-1}$ and choose $x_i$ with $p_{x_i,y}(G_i)$ as large as possible.

We count the number of paths between $x$ and $y$ by summing over possibilities for the second vertex in a path. We get the following inequality

\begin{align*}
p_{x,y}(G) &= \sum_{z \in N(x)}p_{z,y}(G \setminus \{x\})\\
& \le d_G(x_1) \cdot \max\{p_{z,y}(G_2): z \in N(x_1)\} \\
& = d_G(x_1) p_{x_2,y}(G_2).
\end{align*}

Repeating this process gives
$$p_{x_1,y}(G) \le \prod_{i=1}^{\ell} d_{G_i}(x_i),$$
where $\ell$ is minimal such that $\max\{p_{x_{\ell+1},y}(G_{\ell+1}): x_{\ell+1} \in N_{G_{\ell}}(x_{\ell})\} = 1$.

For $1 \le i \le \ell$, let $d_i := d_{G_{i}}(x_{i})$. Note that the $d_i$ are positive integers and that $\sum_{i=1}^{\ell} d_i \le m$. Also note that for any $t \in \{1,\ldots, \ell\}$, we have $$\sum_{i=t}^{\ell} d_i \le e(G_{t}).$$ Therefore, as $G_{t}$ is an $(n-t+1)$-vertex $H$-free graph, $\sum_{i=t}^{\ell} d_i \le \ex(n-t+1;H)$. The result follows by letting $r_i = 0$ for $i = 2,\ldots,n-\ell$ and $r_i = d_{n+1-i}$ for $i = n+1-\ell,\ldots, n$.
\end{proof}

We now prove Lemma \ref{count}.

\begin{proof}[Proof of Lemma \ref{count}]
Following on from the proof of Lemma \ref{ref3count}, we consider a relaxation of the constraints given in the statement of Lemma \ref{ref3count}. Recall that $n_0:= n_0(H)$ is such that $\ex(s;H) = t_k(s)$ and $\ex(s;H) \ge 10s$ for all $s \ge n_0$. We look to maximise
\begin{equation}
\label{prod2}
\prod_{i=2}^{n} \max\{r_i,1\},
\end{equation}
under the following relaxed constraints: 
\begin{enumerate}
\item[(a)] $r_i \in \mathbb{Z}_{\ge 0}$, for $i > n_0,$
\item[(b)] $r_i \in \bR_{\ge 0}$, for $i \le n_0,$
\item[(c)] $\sum_{i=2}^n r_i \le m,$ and
\item[(d)] $\sum_{i=2}^t r_i \le \ex(t;H),$ for each $n_0 \le t \le n.$
\end{enumerate}

Since $m \ge \beta(H) n$, we have $\frac{m}{n} \ge \frac{10t_k(n_0)}{n_0-1}$. Now let $(r_i)_{i=2}^n$ be a sequence maximising \eqref{prod2} subject to (a)-(d). We may assume that $r_2,\ldots,r_{n_0}$ and $r_{n_0+1},\ldots,r_n$ are in increasing order as this will not violate (a)-(d). 
\begin{claim}
There is some $I \in [n_0+1,n-2]$ such that:
\begin{itemize}
\item[(i)] $r_i = \frac{t_k(n_0)}{n_0-1}$, for $i \le n_0$,
\item [(ii)] $r_i = t_k(i) - t_k(i-1)$, for $n_0 +1 \le i \le I$, and
\item [(iii)] $r_i \in \{r_I, r_I + 1\}$, for $i > I$. 
\end{itemize}
\end{claim}
\begin{proof}[Proof of Claim]
Let $T = \sum_{i=2}^{n_0}r_i$. Then $(r_2,\ldots,r_{n_0}) = (0,\ldots,0,\frac{T}{S},\ldots,\frac{T}{S})$ for some $S \in [n_0-1]$ (or else we can increase $\prod_{i=2}^{n_0}r_i$). We may assume that $T$ is an integer as we can replace $T$ by $\lceil T \rceil$ and still satisfy (a)-(d). Differentiation of the function $j(x) = \left(\frac{T}{x}\right)^x$ shows that if $T \ge en_0$, then $S = n_0-1$ and so $r_i = \frac{T}{n_0-1}$ for each $i \in [n_0]$.

Suppose that $T < e\cdot n_0$. Then since $\frac{m}{n} \ge \beta(H) \ge \frac{10t_k(n_0)}{n_0-1}$, there must be a $j > n_0$ such that $r_j \ge \frac{t_k(n_0)}{n_0-1} \ge 10$. Choose $j$ to be minimal with this property. It can easily be verified that increasing $r_2$ by $2$ and decreasing $r_j$ by $2$ gives a sequence which satisfies (a)-(d) but gives a larger product. Therefore it must be the case that $T \ge e\cdot n_0$ and so $S = n_0-1$.

Now suppose that (i) doesn't hold and so $e \cdot n_0 \le T < t_k(n_0)$. Since $\frac{m}{n} \ge \frac{10t_k(n_0)}{n_0-1}$, there exists some $j > n_0$ such that $r_j > \frac{5t_k(n_0)}{n_0-1}$. Choose $j$ to be minimal with this property and define $(s_i)_{i=2}^n$ by $s_i = \frac{T+1}{n_0-1}$ for $i \le n_0$, $s_j = r_j -1$ and $s_i = r_i$ otherwise. Then $(s_i)_{i =2}^n$ is a sequence satisfying (a)-(d) which gives a larger product, a contradiction. Therefore $T = t_k(n_0)$ and (i) holds.

Now suppose that (ii) does not hold and so $r_{n_0+1} < t_k(n_0+1)-t_k(n_0)$. Since $\frac{m}{n} \ge 2(t_k(n_0+1)-t_k(n_0))$, there must be a $j > n_0$ such that $r_j > t_k(n_0+1)-t_k(n_0)$. Choose $j$ to be minimal with this property and define $(s_i)_{i=2}^n$ by $s_{n_0+1} = r_{n_0+1}+1$, $s_j = r_j-1$ and $s_i = r_i$ otherwise. Then $(s_i)_{i =2}^n$ is a sequence satisfying (a)-(d) which gives a larger product, a contradiction. Therefore $r_{n_0+1} = t_k(n_0+1)-t_k(n_0)$ and (ii) holds.

Let $j > n_0+1$ be minimal such that $\sum_{i=1}^j r_i \le t_k(j)-1$ (such a $j$ must exist since $m < t_k(n)$) and set $I=j-1$. If (iii) does not hold then there exists some $t \ge j$ such that $r_j + 1 < r_t$. Let $t$ be minimal with this property, and define $s_j := r_j + 1$, $s_t := r_t - 1$, and $s_i := r_i$ for all $i \not\in \{j,t\}$. The sequence $(s_i)_{i \in [n]}$ satisfies (a)-(d) but
	\begin{align}
		\prod_{i =2}^n \max\{r_i,1\} < \prod_{i =2}^n \max\{s_i,1\}, \nonumber
	\end{align}
a contradiction. Therefore $(r_i)_{i =1}^n$ satisfies properties (i)-(iii), completing the proof of the claim.

Finally note that $I\le n-2$ follows from $m \le t_k(n) - 10n.$
\end{proof}
Putting the values for $r_i$ from the claim into (\ref{prod2}), we see that
\begin{align}
\label{prodsi}
p_{x,y} &\le  \left(\frac{t_k(n_0)}{n_0-1}\right)^{n_0-1} \prod_{i= n_0 + 1}^{I}[t_k(i) - t_k(i-1)] \prod_{i=I+1}^n r_i \nonumber \\
& = O\left(\prod_{i=2}^n s_i\right),
\end{align}
where $(s_i)$ is some sequence such that $s_i = t_k(i) - t_k(i-1)$ for $i \in \{2,\ldots, I\}$, $s_i \in \{s_I,s_I+1\}$ for $i > I$, and $m = \sum_{i=2}^n s_i$.

Note that $s_i = t_k(i) - t_k(i-1) =  (i-1) - \left\lfloor \frac{i-1}{k}\right\rfloor$ for $i \le I$. Then the sequence $(s_i)_{i=2}^I$ is just the natural numbers up to $I-1-\left\lfloor \frac{I-1}{k} \right\rfloor$ with a repetition at each multiple of $k-1$. In other words, $$\left\{s_i : i \in \left\{2,\ldots,I\right\}\setminus \left\{\ell k +1: \ell \le \frac{I-1}{k}\right\}\right\} = \left[I-1-\left\lfloor \frac{I-1}{k} \right\rfloor\right]$$ and $s_{\ell k+1} = \ell(k-1)$ for each $\ell \le \frac{I-1}{k}$. Letting $b = \left\lfloor \frac{I-1}{k}\right\rfloor$ we have
\begin{equation}
\label{si}
\prod_{i=2}^I s_i= (s_I)! \prod_{j=1}^b j(k-1) = s_I! b! (k-1)^b.
\end{equation}

The remaining $n-I$ elements of the product $\prod_{i=2}^n s_i$ are all at most $s_I+1$. Therefore, by \eqref{prodsi} and \eqref{si} we have
\begin{align}
	p_{x,y} &= O\left(\prod_{i=2}^n s_i\right) \nonumber \\
	&= O\left(s_I!b!(k-1)^b(s_I +1)^{n-I}\right) \nonumber \\
	&= O\left(s_I!b!(k-1)^b s_I^{n-I} e^{\frac{n}{s_I}}\right). \label{before1}
\end{align}

Applying Stirling's approximation and simplifying, \eqref{before1} yields
\begin{align}
	p_{x,y} = O\left(s_I^{n + s_I + 1/2 - I} b^{b+1/2}(k-1)^b \exp\left\{\frac{n}{s_I} - I\right\}\right). \nonumber
\end{align}

Since $s_I = I-1 - \left\lfloor \frac{I-1}{k}\right\rfloor \ge (k-1)\frac{I-1}{k}$ and $b = \left\lfloor \frac{I-1}{k}\right\rfloor \le \frac{I-1}{k}$, we have $b \le \frac{s_I}{k-1}$. Therefore
\begin{align}
	p_{x,y} &= O\left(s_I^{n - b - 1/2} \left(\frac{s_I}{k-1}\right)^{b+1/2}(k-1)^b \exp\left\{\frac{n}{s_I} - I\right\}\right) \nonumber \\
	&= O\left(s_I^n\exp\left\{\frac{n}{s_I} - I\right\}\right). \nonumber
\end{align}

Note that $s_I \in \left[\frac{k-1}{k}(I-1),\frac{k-1}{k}I\right]$ and so
\begin{align}
	p_{x,y} &= O\left((I-1)^n\left(\frac{k-1}{k}\right)^n \left(1+\frac{1}{I-1}\right)^n\exp\left\{\frac{kn}{(k-1)(I-1)} - (I-1)\right\}\right) \nonumber \\
	&= O\left((I-1)^n\left(\frac{k-1}{k}\right)^n \exp\left\{\frac{(2k-1)n}{(k-1)(I-1)} - (I-1)\right\}\right). \nonumber
\end{align}

Substituting $I-1 = \alpha n$ gives 
\begin{align}
	p_{x,y} &= O\left(\alpha^{n}n^{n}\left(\frac{k-1}{k}\right)^ne^{\frac{2k-1}{(k-1)\alpha} - \alpha n}\right). \label{alphaexp}
\end{align}

It remains to determine the value of $\alpha$. We do this by counting edges. Since $m = \sum_i s_i$, we see that 
\begin{align}
m \ge t_k(I) + s_I\left(n-I\right). \label{alpha}
\end{align}

Arguing as for \eqref{si}, we see that
\begin{align*}
	t_k(I) &= \sum_{i=1}^{s_I} i + (k-1)\sum_{j=1}^b j  \\
		&= \frac{1}{2}(s_I^2 + s_I + (k-1)(b^2+b)).
\end{align*}

If we put this value for $t_k(I)$ into \eqref{alpha} we see that
\begin{align*}
m &\ge \frac{1}{2}(s_I^2 + s_I + (k-1)(b^2+b))+ s_I\left(n-(I-1)\right) - s_I \\
&=\frac{1}{2}\left(s_I^2 - s_I + (k-1)(b^2+b)\right)+ s_I\left(n-(I-1)\right).
\end{align*}

Now consider that $b = \left\lfloor \frac{I-1}{k} \right\rfloor \ge \frac{I-1}{k} -1,$ so that $b^2+b \ge \left(\frac{I-1}{k}\right)^2 - \frac{I-1}{k}.$ Recall also that $s_I \ge \tfrac{k-1}{k}(I-1)$ and so

\begin{align*}
m &\ge \frac{1}{2}\left(\left(\frac{k-1}{k}\right)^2(I-1)^2 - \frac{k-1}{k}(I-1) +\frac{k-1}{k^2}(I-1)^2 - \frac{k-1}{k}(I-1)\right) \\
&+\frac{k-1}{k}(I-1)n - \frac{k-1}{k}(I-1)^2 \\
&\ge \frac{k-1}{k}n(I-1)-\frac{k-1}{2k}(I-1)^2 - 3\frac{k-1}{k}(I-1).
\end{align*}

Substituting $(I-1) = \alpha n$ and rearranging gives

\begin{align*}
\left(\left(1-\frac{3}{n}\right) - \alpha\right)^2 &\ge \left(1-\frac{3}{n}\right)^2 - \frac{2k}{k-1}\frac{m}{n^2}.
\end{align*}

Recall that $I \le n-2$ and so $\alpha \le \left(1-\frac{3}{n}\right)$. On the other side of the inequality, $\left(1-\frac{3}{n}\right)^2 - \frac{2k}{k-1}\frac{m}{n^2}$ is positive since $m \le t_k(n) - 10n$. Therefore we can take square roots and rearrange to get

\begin{align*}
	\alpha &\le \left(1-\frac{3}{n}\right) - \left(\left(1-\frac{3}{n}\right)^2 - \frac{2k}{k-1}\frac{m}{n^2}\right)^{\frac{1}{2}} = \left(1-\frac{3}{n}\right)\lambda.
\end{align*}

Since the expression $\alpha^{n}n^{n}\left(\frac{k-1}{k}\right)^ne^{\frac{2k-1}{(k-1)\alpha} - \alpha n}$ is increasing in $\alpha$ when $\alpha \le 1 -\frac{2}{n}$, \eqref{alphaexp} is maximised by setting $\alpha = \left(1-\frac{3}{n}\right)\lambda$. We are then done since

\begin{align*}
	\left(1-\frac{3}{n}\right)^n\lambda^{n}n^{n}\left(\frac{k-1}{k}\right)^ne^{\frac{2k-1}{(k-1)\left(1-\frac{3}{n}\right)\lambda} - \left(1-\frac{3}{n}\right)\lambda n} = O\left(\lambda^{n}n^{n}\left(\frac{k-1}{k}\right)^ne^{\frac{2k-1}{(k-1)\lambda} - \lambda n}\right).
\end{align*}

\end{proof}

Theorem~\ref{cyclecount} follows easily from the idea of this proof by applying the following theorem of Erd\H{o}s and Simonovits.

\begin{thm}[Erd\H{o}s and Simonovits {\cite[Theorem 1]{Sim2}}]
\label{simthm2}
Let $H$ be a graph with $\chi(H) = k+1$. Then, 
$$\lim_{n \rightarrow \infty} \frac{\ex(n;H)}{\binom{n}{2}} = 1 - \frac{1}{k}.$$
\end{thm}

\begin{proof}[Proof of Theorem \ref{cyclecount}]
Let $\varepsilon >0$. By Theorem \ref{simthm2} and the fact that $t_k(n) \sim \left(1 - \frac{1}{k}\right)\binom{n}{2}$, we know that for $n$ sufficiently large, $\ex(n;H) \le (1 + \varepsilon)t_k(n)$. Thus, for $n$ sufficiently large, $\ex(s;H) \le (1 + \varepsilon)t_k(s)$ for all $n^{\frac{1}{2}} \le s \le n$. For ease of notation, let $n_1 := n^{\frac{1}{2}}$.

To bound the number of cycles in the graph, we wish to bound $p_{x,y}(G)$ for $x,y \in V(G)$. From Lemma~\ref{ref3count}, we see that it is enough to bound the product $$\prod_{i=2}^n \max\{r_i,1\},$$ where $(r_i)$ satisfies the relaxed conditions:
\begin{enumerate}
\item[(i)] $r_i \in \bR^+$, for all $i$, and
\item[(ii)] $\sum_{i=2}^t r_i \le (1+ \varepsilon)t_k(t)$, for each $n_1 \le t \le n$.
\end{enumerate}

It is easily seen that this expression is maximised when $r_i :=\frac{(1+\varepsilon)t_k(n_1)}{n_1 - 1}$ for $i=2,\ldots,n_1$ and $r_i = (1+\varepsilon)(t_k(i)-t_k(i-1))$ otherwise. Therefore, we arrive at the following bound:
\begin{align}
\prod_{i=2}^n r_i &\le \left(\frac{(1+\varepsilon)t_k(n_1)}{n_1 - 1}\right)^{n_1 - 1}\prod_{i=n_1+1}^n(1+\varepsilon)(t_k(i)-t_k(i-1)) \nonumber \\
&= O\left(e^{n_1}\prod_{i=2}^n(1+\varepsilon)(t_k(i)-t_k(i-1))\right) \nonumber \\
&=  O\left(e^{\varepsilon n + n_1} \prod_{i=2}^n(t_k(i)-t_k(i-1))\right). \label{nuisance43}
\end{align}

Recall from \eqref{si} that, defining $b = \left\lfloor \frac{n-1}{k}\right\rfloor,$ we have
\begin{align*}
	\prod_{i=2}^n(t_k(i)-t_k(i-1)) &= (n-1-b)!b!(k-1)^b.
\end{align*}

Applying Stirling's approximation and simplifying gives
\begin{align*}
	\prod_{i=2}^n(t_k(i)-t_k(i-1)) &= O\left((n-1-b)^{n-1-b + 1/2}b^{b+1/2}e^{-n}(k-1)^b\right) \\
	&= O\left(\left(\frac{k-1}{k}\right)^n n^{n+1}e^{-n}\right).
\end{align*}

Putting this into \eqref{nuisance43} gives
\begin{equation}
\label{pathcount}
p_{x,y} = O\left(\left(\frac{k-1}{k}\right)^n n^{n+1} e^{\varepsilon n + n_1 - n}\right).
\end{equation}
Now, as in the proof of Lemma \ref{easycor}, we see that by \eqref{pathcount} and the fact that $n_1 = o(n)$,
\begin{align*}
c(G) &\le \sum_{xy \in E(G)} p_{x,y}\\
&= O\left(n^2 \left(\frac{k-1}{k}\right)^n n^{n+1} e^{\varepsilon n + n_1 - n}\right)\\
&= O\left(\left(\frac{k-1}{k}\right)^nn^ne^{-(1-2\varepsilon)n}\right).
\end{align*}
Since $\varepsilon$ is arbitrary, we have our result.
\end{proof}

\section{Proof of Theorem \ref{main}}\label{secmain}
Here we complete the proof of Theorem~\ref{main}. This will follow from the next two lemmas. 

The first gives a lower bound on the number of edges in an extremal graph. (See also \cite[Theorem 5.3.2]{arman} for a $K_{k+1}$ version.)

\begin{lem}\label{edgecount}
Let $H$ be a graph $\chi(H) = k+1 \ge 3$ containing a critical edge. For sufficiently large $n$, let $G$ be an $n$-vertex $H$-free graph with $m$ edges and $c(G) \ge c(T_k(n))$. Then $m \ge \frac{n^2(k-1)}{2k} - O\left(n\log^2(n)\right)$.
\end{lem}

Given this lemma, we can apply Theorem~\ref{stable} to show that any extremal graph $G$ is close to being $k$-partite. We then carefully count the number of cycles in such a graph. In what follows, for a graph $G$ and a $k$-partition of its vertices, we call edges within a vertex class \emph{irregular} and those between vertex classes \emph{regular}. Define a \emph{best} $k$-partition of a graph $G$ to be one which minimises the number of irregular edges contained within $G$. The next lemma counts the cycles using only regular edges if $G$ is not $T_k(n)$. Recall that $c_r(G)$ is the number of cycles of length $r$ in $G$.

\begin{lem}\label{regcycle}
Let $H$ be a graph with $\chi(H) = k+1 \ge 3$ containing a critical edge. Suppose $G\not\cong T_k(n)$ is an $n$-vertex $H$-free graph with $c(G) \ge c(T_k(n))$. Then for sufficiently large $n$, the number of cycles using only regular edges in the best $k$-partition of $G$ is at most:
$$\left\{
\begin{array}{c l}      
    c(T_k(n)) - \frac{1}{16k}h(T_k(n)) & \text{ for } k\ge 3,\\
    c(T_2(n)) - \frac{1}{8}c_{2\lfloor \frac{n}{2} \rfloor}(T_2(n)) & \text{ for } k=2.
\end{array}\right.$$
\end{lem}

Given Lemmas~\ref{edgecount} and \ref{regcycle}, we now complete the proof of Theorem~\ref{main}. We will then prove the lemmas themselves. The main work remaining for Theorem~\ref{main} is to count the number of cycles using irregular edges. 

\begin{proof}[Proof of Theorem \ref{main}]
Let $H$ be a graph with a critical edge with chromatic number $\chi(H) = k+1 \ge 3$, and suppose $G$ is an $n$-vertex $H$-free graph with $c(G) = m(n;H)$. Then, in particular, $c(G) \ge c(T_k(n))$. Suppose for a contradiction that $G$ is not isomorphic to $T_k(n)$. Fix a best $k$-partition of $G$: by Lemma \ref{edgecount} and Theorem \ref{stable}, we know that for sufficiently large $n$, the graph $G$ has at most $n^{0.55}$ irregular edges in its best $k$-partition.

Let $c^I(G)$ be the number of cycles in $G$ containing at least one irregular edge and let $c^R(G)$ be the number of cycles in $G$ using only regular edges. If $c^I(G) = o(h(T_k(n))$, then by applying Lemma~\ref{regcycle} and taking $n$ sufficiently large, we have $c(G) = c^R(G) + c^I(G) < c(T_k(n))$. Thus $c^I(G) = \Omega(h(T_k(n))).$

Let $E_I$ be the set of irregular edges in $G$. For each non-empty $A \subseteq E_I$, let $C_A$ be the set of cycles $C$ in $G$ such that $E(C) \cap E_I = A$ and such that $C$ contains at least one regular edge. Fix $A$ such that $C_A$ is non-empty and fix an edge $a_1 a_2 \in A$. (Note that $A$ must be a vertex-disjoint union of paths or else it would not be possible to have a cycle using all edges in $A$.) For any cycle $C = x_1 x_2 \cdots x_j$ in $C_A$, with $x_1 = a_1$ and $x_2 = a_2$, define $S(C)$ to be the directed cycle $x_1x_2\cdots x_j$ (so for all $i$, the edge $x_ix_{i+1}$ is directed towards $x_{i+1}$, where indices are taken modulo $j$).
 
For each $C \in C_A$, the orientation of $S(C)$ induces an orientation $f_C$ on the edges of $A$. Given a fixed orientation $f$ of $A$, we write
 $$C_A(f):= \left\{C \in C_A: f_C = f\right\}.$$
 We will bound the size of each $C_A(f)$. A bound on $c^I(G)$ will then follow by summing over all possible $A$ and $f$.  
 
Let $G/A$ be the graph obtained by contracting every edge in $A$. Then remove the remaining irregular edges to form $J$ (so $J$ is an $H$-free $k$-partite graph with $n-|A|$ vertices, as $A$ is a vertex-disjoint union of paths, and each edge of $A$ lies inside some vertex class of our $k$-partition). For each cycle $C$ in $C_A(f)$, we obtain an oriented cycle $g(C)$ in $J$ by replacing each maximal path $u_1\cdots u_j$ in $E(C) \cap A$ oriented from $u_1$ to $u_j$ by $u_1$. As $C$ contains at least one regular edge, $g(C)$ is either an edge or cycle in $J$.
 
We claim that $g$ is injective on $C_A(f)$. Indeed suppose that there exists a cycle $C \in C_A(f)$. Recall that $A$ is a vertex-disjoint union of paths and furthermore that $f$ orients the paths of $A$. Denote these oriented paths $\left(u^1_i\right)_{i \in [\ell_1]},\ldots,\left(u^t_i\right)_{i \in [\ell_t]}$. Each cycle $C \in C_A(f)$ must contain these oriented paths as segments (each edge of $A$ must be contained in $C$ and it is not possible to break up a path or else a vertex must be adjacent to more than two edges in the cycle). Therefore we have an inverse of $g$ which takes a cycle from $g\left(C_A(f)\right)$ and replaces each instance of $u^j_1$ with the path $u^j_1\cdots u^j_{\ell_j}$.

As $J$ is a $k$-partite graph on $n-|A|$ vertices, by Lemma~\ref{Turanbest} we have 
$$c(J) \le c(T_k(n-|A|)).$$
Recall that for each $C \in C_A(f)$, $g(C)$ is either an edge or a cycle in $J$. We therefore have
$$|C_A(f)| \le 2\cdot c(T_k(n-|A|)) + 2|E(T_k(n))| \le 4 \cdot c(T_k(n-|A|)),$$
for sufficiently large $n$ by applying Lemma~\ref{kKMain} and recalling that $|A| \le n^{0.55}$.
Let $F_A$ be the set of all possible orientations $f$ of $A$. We have
\begin{equation}
\label{irreg}
c^I(G) \le \left|E_I\right|^{|E_I|} + \sum_{A\subseteq E_I}\sum_{f \in F_A}|C_A(f)|,
\end{equation}
where the first term counts cycles that contain only irregular edges and the second term counts cycles in $c^I(G)$ that contain both a regular and irregular edge.

We will bound the second term of this expression. Recalling that there are at most $n^{0.55}$ irregular edges, we get that
$$\sum_{A\subseteq E_I}\sum_{f \in F_A}|C_A(f)| \le \sum_{i=1}^{n^{0.55}}{\binom{n^{0.55}}{i}}2^i\cdot 4 \cdot c(T_k(n-i)).$$

For $k \ge 3$, we now apply Lemma \ref{secondcount} and Lemma \ref{recursion} for each $i$ in the sum,
	\begin{align}
		\sum_{A\subseteq E_I}\sum_{f \in F_A}|C_A(f)|&\le \sum_{i=1}^{n^{0.55}}{\binom{n^{0.55}}{i}}e^{\frac{2k}{k-2}}2^{i+2}h(T_k(n-i)) \nonumber \\
		&\le 4e^{\frac{2k}{k-2}}\sum_{i=1}^{n^{0.55}}{\binom{n^{0.55}}{i}}\left(\frac{2k}{k-2}\right)^i\frac{h(T_k(n))}{(n-1)_i} \nonumber \\
		&\le 4e^6 h(T_k(n)) \sum_{i\ge 1} n^{0.55i}\left(\frac{6}{n-n^{0.55}}\right)^i \nonumber \\
		&= o\left(h(T_k(n))\right). \nonumber
	\end{align}
We have $|E_I|^{|E_I|} \le (n^{0.55})^{n^{0.55}}$ which is $o(h(T_k(n)))$ by Lemma~\ref{kKMain}. Therefore, using \eqref{irreg} we see that $c^I(G) = o(h(T_k(n))$, a contradiction. Therefore $G$ is isomorphic to $T_k(n)$.

Similarly for $k=2$, we apply Lemma \ref{second2count} to get
	\begin{align}
		\sum_{A\subseteq E_I}\sum_{f \in F_A}|C_A(f)|&\le \sum_{i=1}^{n^{0.55}}{\binom{n^{0.55}}{i}}2^i\cdot 8e \cdot \left(\frac{4}{n}\right)^i c_{2\lfloor n/2 \rfloor}(T_2(n)) \nonumber \\
		&\le 8e \cdot c_{2\lfloor n/2 \rfloor}(T_2(n)) \sum_{i = 1}^{n^{0.55}}n^{0.55i}\left(\frac{8}{n}\right)^i \nonumber \\
		&= o\left(c_{2\lfloor n/2 \rfloor}\left(T_2(n)\right)\right), \nonumber
	\end{align}
and we conclude as before.
\end{proof}

We now present the proofs of Lemmas~\ref{edgecount} and \ref{regcycle}.

\begin{proof}[Proof of Lemma~\ref{edgecount}]
First suppose that $m = O(n)$. We can then crudely bound $p_{x,y}(G)$ by Lemma \ref{ref3count}. By \eqref{prod} and constraints (i) and (ii) above we have
	\begin{align}
		p_{x_1,y}(G) \le \max_{\ell}\prod_{i=1}^{\ell}r_i \le \max_{\ell}\left(\frac{m}{\ell}\right)^{\ell}. \nonumber
	\end{align}	

The function $f(x) = \left(\frac{m}{x}\right)^x$ is maximised at $x = \frac{m}{e}$ and so $p_{x_1,y}(G) \le e^{\frac{m}{e}} = e^{O(n)}$. This is asymptotically smaller than $c(T_k(n))$ by Lemma~\ref{kKMain}.

So $m \not= O(n)$. Suppose that $m \le t_k(n)-10n$ (otherwise we are done so assume) so that we obtain a bound for $c(G)$ from Corollary \ref{easycor}. Dividing this bound by $c(T_k(n)) = \Omega((\frac{k-1}{k})^nn^{n-\frac{1}{2}}e^{-n})$ gives
	\begin{align}
		\frac{c(G)}{c(T_k(n))} = O\left(\lambda^{n}n^{2.5}e^{\frac{2k-1}{(k-1)\lambda} + \left(1-\lambda\right) n}\right), \label{nuisance}
	\end{align}
where $\lambda$ is defined in \eqref{alphadef}.

If we take the logarithm of the right hand side and call it $R$ for ease of notation, we get
	\begin{align}
		R &\le 2.5\log(n) + n(\log(\lambda)+(1-\lambda)) + \frac{2k-1}{(k-1)\lambda} + O(1) \nonumber \\
		&\le 2.5\log(n) + n(\log(\lambda)+(1-\lambda)) + 3\lambda^{-1} + O(1). \nonumber
	\end{align}
	
First assume that $\lambda \le 1- n^{-\frac{1}{2}}\log(n)$: we will show that then $R \rightarrow -\infty$ and so \eqref{nuisance} is $o(1)$.

If $\lambda \le e^{-2}$, then $\log(\lambda) + (1-\lambda) \le \frac{\log(\lambda)}{2}$. Furthermore we see from \eqref{alphadef} that $\lambda = \Omega\left(\frac{m}{n^2}\right)$ and so $\lambda^{-1} = o(n)$. Therefore
	\begin{align}
		R &\le 2.5 \log(n) + \frac{n}{2}\log(\lambda) + o(n) \nonumber \\
		&\le 2.5 \log(n) - n + o(n) \rightarrow -\infty, \nonumber	
	\end{align}
as $n$ tends to infinity.

Otherwise, $\lambda^{-1} \le e^2$ and since (by assumption) $\lambda \le 1- n^{-\frac{1}{2}}\log(n)$, we may apply Taylor's theorem to see
	\begin{align}
		R &\le 2.5\log (n) - n(1-\lambda)^2 + 3e^2 \nonumber \\
		&\le 2.5\log(n) - \log^2(n) + 3e^2 \rightarrow -\infty, \nonumber
	\end{align}
as $n$ tends to infinity.

In either case $R$ tends to $-\infty$ for sufficiently large $n$, and we must have that $c(G) < c(T_k(n))$, a contradiction.

Therefore $\lambda > 1 - \log(n)n^{-\frac{1}{2}}$. Equation \eqref{alphadef} now allows us to conclude that $m \ge t_k(n) - O\left(n\log^2(n)\right)$, as required.
\end{proof}

For the proof of Lemma~\ref{regcycle} we require the Erd\H{o}s-Stone Theorem \cite{erd-stone}.

\begin{thm}[Erd\H{o}s-Stone \cite{erd-stone}]
\label{erdstone}
Let $k \ge 2$, $t \ge 1$, and $\varepsilon > 0$. Then for $n$ sufficiently large, if $G$ is a graph on $n$ vertices with
$$ e(G) \ge \left(1 - \frac{1}{k-1} + \varepsilon\right)\binom{n}{2},$$ 
then $G$ must contain a copy of $T_k(kt)$.
\end{thm}

We now apply this theorem to complete the proof of Lemma~\ref{regcycle}.

\begin{proof}[Proof of Lemma~\ref{regcycle}]
Let the best $k$-partition of $G$, be $V_1,\ldots,V_k$. By Lemma~\ref{edgecount}, $e(G) > t_k(n) - O\left(n \log^2 n\right)$, and so Theorem \ref{stable} tells us that $G$ contains $t_k(n)(1-o(1))$ edges between its vertex classes $V_1,\ldots,V_k$. We therefore have $|V_i| = \tfrac{n}{k}(1+o(1))$ for each $i$. Also note that $G$ cannot be $k$-partite (else $c(G) < c(T_k(n))$ by Lemma \ref{Turanbest}). Therefore $G$ must contain an irregular edge. Now we count the cycles in $G$ which contain only regular edges. Note that if we define $G^R$ to be $G \backslash E_I$, where  $E_I$ is the set of irregular edges, then $G^R$ is $k$-partite; $G^R \subseteq K_{\underline{a}}$ for some $\underline{a} = (a_1,\ldots, a_k) \in \mathbb{N}^k$.

Let $t$ be such that $H  \subseteq T_k(tk) + e$, where $e$ is any edge inside a vertex class of $T_k(tk)$. Pick an irregular edge $uv$: without loss of generality we may assume $uv \in V_1$. We first show that $u$ and $v$ cannot have $\frac{n}{10k}$ common neighbours in every other vertex class. Suppose otherwise and form a set $Q$ by picking $\frac{n}{10k}$ vertices in $N(u)\cap N(v) \cap V_i$ for $i = 2, \ldots, k$ and picking $\frac{n}{10k}$ vertices in $V_1$ to be in $Q$.

The graph $G^R[Q]$ does not contain a copy of $T_k(tk)$: if it did, it would contain a copy of $T_k(tk)+e$ and hence a copy of $H$. So then applying Theorem \ref{erdstone}, there are $\Omega(n^2)$ regular edges that are not present in $G$, a contradiction. Thus, without loss of generality, $|N(u) \cap N(v) \cap V_2| < \frac{n}{10k}$ and, again without loss of generality, $|N(v) \cap V_2| \le \frac{5n}{8k}$ (since $|V_2| = \tfrac{n}{k}(1+o(1))$ and we may assume that $n$ is large). 

When $k \ge 3$, this means that $G$ cannot contain at least $\frac{3}{8}$ of the Hamilton cycles contained in $K_{\underline{a}}$ which start from $v$ and then go to vertex class $V_2$. Recall that $h_v(i,K_{\underline{a}})$ is the number of permutations of $V(K_{\underline{a}}) = \{v_1, \ldots, v_n\}$ such that $v_1 = v$, $v_2 \in V_i$ and $v_1 \cdots v_n$ is a Hamilton cycle. Since cycles may be counted at most twice due to orientation when considering permutations, the number of Hamilton cycles in $K_{\underline{a}}$ which start from $v$ and then go to vertex class $V_2$ is at least $\tfrac{1}{2}h_v(2,K_{\underline{a}})$. By applying \eqref{first1}, we get
 	\begin{align}
		c(G^R) &\le c(K_{\underline{a}}) - \frac{3}{8}\cdot \frac{1}{2}h_v(2,K_{\underline{a}}) \nonumber \\
		&= \sum_{r=3}^{n-1}c_r(K_{\underline{a}}) + \frac{1}{2}\sum_{i=3}^kh_v(i,K_{\underline{a}}) +  \left(\frac{1}{2} - \frac{3}{16}\right)h_v(2,K_{\underline{a}}). \nonumber
	\end{align}

Let $\underline{b} = (b_1,\ldots,b_n)$, be such that $b_i \ge b_j$ if and only if $a_i \ge a_j$, and that $K_{\underline{b}} \cong T_k(n).$ Recall that $a_i = \frac{n}{k}(1+o(1))$ and so $\prod_{i =1}^ke^{\left|\log\left(\frac{b_i}{a_i}\right)\right|} = (1+o(1))$. Therefore by applying Lemmas \ref{close} and \ref{Turancount} we get
	\begin{align}
		c(G^R) &\le \sum_{r=3}^{n-1}c_r(K_{\underline{a}}) + \prod_{i =1}^ke^{\left|\log\left(\frac{b_i}{a_i}\right)\right|}\left[\frac{1}{2}\sum_{i=3}^kh_v(i,T_k(n)) +  \left(\frac{1}{2} - \frac{3}{16}\right)h_v(2,T_k(n))\right] \nonumber \\
		&= \sum_{r=3}^{n-1}c_r(K_{\underline{a}}) + (1+o(1))\left(c_n(T_k(n)) - \frac{3}{16}h_v(2,T_k(n))\right) \nonumber \\
		&\le (1+o(1))\left(c(T_k(n)) - \frac{1}{8k}h(T_k(n))\right). \nonumber
	\end{align}
Finally, we can apply Lemma \ref{secondcount} to get
	\begin{align}
		c(G^R) &\le (1+o(1))\left(c(T_k(n)) - \frac{1}{24k}h(T_k(n)) - \frac{1}{12k}h(T_k(n))\right) \nonumber \\
		&\le (1+o(1))\left(c(T_k(n))\left(1-\frac{e^{-\frac{2k}{k-2}}}{24k}\right) - \frac{1}{12k}h(T_k(n))\right), \nonumber
	\end{align}
and so for $n$ sufficiently large, $c(G^R) \le c(T_k(n)) -\frac{1}{16k}h(T_k(n))$.
		
For $k=2$, first consider that if $|V_1|$ and $|V_2|$ differ in size by more than $1$, then $G^R$ contains no cycle of length $2\lfloor n/2 \rfloor$. Counting cycles by length and applying Lemma \ref{Turanbest} gives
	\begin{align}
		c(G^R) &= \sum_{r=2}^{\lfloor n/2 \rfloor -1} c_{2r}(G^R) \nonumber \\
		&\le \sum_{r=2}^{\lfloor n/2 \rfloor -1} c_{2r}(T_2(n)) \nonumber \\
		&= c(T_2(n)) - c_{2\lfloor n/2 \rfloor}(T_2(n)). \nonumber
	\end{align}

Therefore assume that $|V_1|$ and $|V_2|$ differ in size by at most $1$ (so $G^R$ is a subgraph of $T_2(n)$). Recall (from the third paragraph of this proof) that $G^R$ contains a vertex $v$ with degree at most $5n/16$. Therefore, when applying the argument for $k\ge 3$, we lose at least a quarter of the cycles of length $2\lfloor n/2 \rfloor$ which contain $v$ from $T_2(n)$. Note that $v$ is present in at least half of the cycles of length $2\lfloor n/2 \rfloor$ in $T_2(n)$ and so $c(G^R) \le c(T_2(n)) - \frac{1}{8}c_{2\left\lfloor \frac{n}{2} \right\rfloor}(T_k(n))$.
\end{proof}

\section{Counting Cycles in Complete multi-partite Graphs}\label{sectech}
In this section we present the proofs for the lemmas concerning counting cycles in complete multi-partite graphs that we stated in Section \ref{seckpart}. We start with some preliminary lemmas. In order to state these we require some technical definitions.

 Define a \emph{code} on an alphabet $\cA$ to be a string of letters $a_1\cdots a_n$ where each $a_i$ is in $\cA$. For $k \ge 3$, we now discuss a way to count the number of Hamilton cycles in a $k$-partite graph $G$. Suppose each vertex class $V_i$ of $G$ is ordered. Consider a code $a_1 \cdots a_n$, where each $a_i \in [k]$. From such a code, we attempt to construct a Hamilton cycle $v_1\cdots v_n$ in $G$ as follows: for $j=1,\ldots,n$ let $p(j) := \left|\left\{\ell \le j: a_{\ell} = a_j\right\}\right|$. Define $v_j$ to be the $p(j)$-th vertex in $V_{a_j}$. For $v_1\cdots v_n$ to be a Hamilton cycle, each letter must appear in the code $a_1\cdots a_n$ the correct number of times ($\left|\left\{j : a_j = i\right\}\right| = \left|V_i\right|$, for each $i \in [k]$) and any two consecutive letters of the code must be distinct ($a_j \neq a_{j+1}$ for each $j \in [n-1]$, and $a_1 \neq a_n$). 
 
 For a code $a_1 \cdots a_n$, with each $a_i \in [k]$, we say that the code is in $Q$ if $a_i \neq a_{i+1}$ for each $i$, where indices are taken modulo $n$ (so each pair of consecutive letters are distinct). For $\underline{c} = (c_1,\ldots,c_k) \in \bN^k$, we say that the code is in $P_{\underline{c}}$ if there are $c_i$ copies of $i$, for each $i \in [k]$. Finally we say that a code is in $P_{n,k}$ if it is in $P_{\underline{d}}$, where $\underline{d} = (d_1,\ldots,d_k) \in \bN^k$ is such that $d_1 \le d_2 \le \ldots \le d_k \le d_1 +1$ and $\sum_i d_i = n$.
 
 In what follows it will be useful to consider a random code, so let $C_{n,k}$ denote the random code $C_{n,k} = a_1 \cdots a_n$, where each $a_i$ is independently and uniformly distributed on $[k]$.

Enumerate the vertex set $V(K_{\underline{c}}) = \{v_1,\ldots,v_n\}.$ We can count the number of Hamilton cycles in $K_{\underline{c}}$ by considering the probability that a permutation $\sigma$ of $[n]$ picked uniformly gives a Hamilton cycle $v_{\pi(1)}\ldots v_{\pi(n)}.$ Since we have a choice of orientation and starting vertex, each Hamilton cycle will be counted $2n$ times, and so
	\begin{align}
		h(K_{\underline{c}}) = \frac{n!}{2n}\bP\left[v_{\pi(1)}\ldots v_{\pi(n)}\mbox{ is a Hamilton cycle}\right]. \label{hamattempt2}
	\end{align}

For $i \in [n],$ define $b_i \in [k]$ such that $v_{\pi(i)} \in V_{b_i}.$ Then $b_1\cdots b_n$ has the same distribution as $C_{n,k}$ conditioned on the event $\left\{C_{n,k} \in P_{\underline{c}}\right\}.$ Note further that $v_{\pi(1)}\ldots v_{\pi(n)}$ is a Hamilton cycle if and only if $b_1\cdots b_n \in Q.$ Putting these into \eqref{hamattempt2} gives
	\begin{align}
		h(K_{\underline{c}}) &=\frac{n!}{2n}\bP[b_1\cdots b_n \in Q] \nonumber \\
		&= \frac{n!}{2n}\bP[C_{n,k} \in Q | C_{n,k} \in P_{\underline{c}}]. \label{ham4}
	\end{align}

Obtaining good bounds on the probability that a random code is in $Q$ (and similarly in $P_{\underline{c}}$) is relatively easy but approximating the probability of the intersection of the events proves more tricky. The following lemma will help us bound \eqref{ham4} from below, in order to prove Lemma~\ref{kKMain}.

\begin{lem}\label{major}
Let $k \ge 2$ and suppose $C_{n,k} = a_1\cdots a_n$ where the $a_i$ are independent and identically uniformly distributed on $[k]$.  If $\underline{c} = (c_1,\ldots,c_k) \in \bN^k$ is such that $\sum_{i} c_i = n$, then
	\begin{align}
		\bP[C_{n,k} \in Q | C_{n,k}  \in P_{n,k}] \ge \bP[C_{n,k} \in Q |  C_{n,k}  \in P_{\underline{c}}],  \nonumber
	\end{align}
and in particular,
	\begin{align}
		\bP[C_{n,k} \in Q | C_{n,k}  \in P_{n,k}] \ge \bP[C_{n,k} \in Q]. \nonumber
	\end{align}
\end{lem}

\begin{proof}
Let $k \ge 2$ and suppose $\underline{c} = (c_1,\ldots,c_k) \in \bN^k$ is such that $\sum_{i} c_i = n$. Suppose that there exist some $i$ and $j$ such that $c_i \le c_j -2$, and let $\underline{c}'=(c'_1,\ldots,c'_k)$ be such that $c'_i = c_i+1, c'_j = c_j-1$ and $c'_t=c_t$ for $t \neq i,j$. It is sufficient to show that $\bP[C_{n,k} \in Q | C_{n,k}  \in P_{\underline{c}'}] \ge \bP[C_{n,k} \in Q | C_{n,k}  \in P_{\underline{c}}]$ -- we may inductively find an $i$ and $j$ until the $c_a$ differ by at most one and $\underline{c}$ corresponds to the vertex class sizes of a Tur\'{a}n graph.

Fix a subset $A$ of $[n]$ with $|A| = n - (c_i+c_j)$ and let $R_{A,\underline{c}}$ be the event that $C_{n,k}$ is in $P_{\underline{c}}$, that $A = \{\ell : a_{\ell} \neq i,j \}$, and that $a_{\ell} \neq a_{\ell+1}$ for all $\ell$ in $A$ and $a_n \neq a_1$ if both $n$ and $1$ are in $A$. $R_{A,\underline{c}}$ can be thought of as the event that everything in the code except the letters with values $i$ and $j$ behave well. Now note that we can partition over all the sets of size $n - (c_i+c_j)$ in $[n]$, and get the expression
	\begin{align}
		\bP[C_{n,k} \in Q | C_{n,k} \in P_{\underline{c}}] = \sum_{A \in \binom{[n]}{n - (c_i+c_j)}} \bP[C_{n,k} \in Q | R_{A,\underline{c}}] \cdot \bP[R_{A,\underline{c}} | C_{n,k} \in P_{\underline{c}}]. \nonumber
	\end{align}

Note that given $P_{\underline{c}}$ holds, we may as well identify $i$ and $j$ when considering whether $R_{A,\underline{c}}$ holds. As such, $\bP[R_{A,\underline{c}} | C_{n,k} \in P_{\underline{c}}]$ is constant with respect to $c_i$ and $c_j$ with fixed $c_i + c_j$. This in turn, means that $\bP[R_{A,\underline{c}} | C_{n,k} \in P_{\underline{c}}] = \bP[R_{A,\underline{c}'} | C_{n,k} \in P_{\underline{c}'}]$ and so to prove the first statement of the lemma, it is sufficient to show that
	\begin{align}
		\bP[C_{n,k} \in Q | R_{A,\underline{c}}] \le \bP[C_{n,k} \in Q | R_{A,\underline{c}'}], \label{codecount3}
	\end{align}
for each $A \subseteq [n]$, with $|A| = n - (c_i+c_j)$.

Let $A \subseteq [n]$, with $|A| = n - (c_i+c_j)$ and condition on the event $R_{A,\underline{c}}$ (note that we may assume that this event is not null else we have nothing to prove). If we consider $C_{n,k}$ as a code that is a cycle (imagine joining $a_1$ to $a_n$), then the occurrences of $i,j$ form a collection of segments of total length $c_i + c_j$ with $c_i$ copies of $i$ and $c_j$ copies of $j$. Conditioning just on $R_{A,\underline{c}}$, we have choice over where we place the $i$ and $j$ letters in the segments. Since we must have $c_i$ total copies of $i$ in the segments, there are $c_i+c_j \choose c_i$ such choices of placement of the $i$ and $j$ letters. Conditional on $R_{A,\underline{c}}$, the $i$ and $j$ placements are uniformly distributed on these $c_i+c_j \choose c_i$ choices. Conditional on $R_{A, \underline{c}}$, for the code $C_{n,k}$ to be in $Q$, the segments all have to be a string of letters alternating between $i$ and $j$. As such the first letter of a segment dictates the remainder of that segment.

Let the lengths of the $\{i,j\}$-segments of $C_{n,k}$ be $r_1,\ldots,r_m$ and let $s_{\rm odd \rm}$ and $s_{\rm even \rm}$ be the number of odd length $\{i,j\}$-segments and even length $\{i,j\}$-segments respectively. We are then able to compute $\bP[C_{n,k} \in Q | R_{A,\underline{c}}]$ by considering the starting letter of each $\{i,j\}$-segment. Suppose that $t$ of the $s_{\rm odd \rm}$ $\{i,j\}$-segments with odd length start with $i$. Then in the code, there will be $s_{\rm odd \rm}-2t$ more appearances of $j$, than of $i$. Therefore, since $C_{n,k} \in P_{\underline{c}}$, we must have $2t - s_{\rm odd \rm} = c_i-c_j$ and so $t = \tfrac{s_{\rm odd \rm}+c_i-c_j}{2}$. Note that if $s_{\rm odd \rm}+c_i-c_j$ is odd, then $\bP[C_{n,k} \in Q | R_{A,\underline{c}}] = 0$ since $t$ must be an integer (and so we have nothing to prove). Therefore we assume that $s_{\rm odd \rm}+c_i-c_j$ is even in what follows.

We can specify such a code by choosing the starting letter of each even interval arbitrarily and choosing exactly $t$ odd intervals to start with $i$. Comparing this with all possible choices of placements of $i$ and $j$ letters, we obtain
	\begin{align}
		\bP[C_{n,k} \in Q | R_{A,\underline{c}}] &= \frac{2^{s_{\rm even \rm}}{s_{\rm odd \rm} \choose t}}{{c_i+c_j \choose c_i}}  \label{codecount1}, \\
		\bP[C_{n,k} \in Q | R_{A,\underline{c}'}] &= \frac{2^{s_{\rm even \rm}}{s_{\rm odd \rm} \choose t+1}}{{c'_i+c'_j \choose c'_i}} \nonumber \\
		&= \frac{2^{s_{\rm even \rm}}{s_{\rm odd \rm} \choose t+1}}{{c_i+c_j \choose c_i+1}} \label{codecountz}.
	\end{align}

Writing $b = c_j-c_i$ and dividing \eqref{codecount1} by \eqref{codecountz}, we get

	\begin{align}
		\frac{\bP[C_{n,k} \in Q | R_{A,\underline{c}}]}{\bP[C_{n,k} \in Q | R_{A,\underline{c}'}]} &= \frac{c_j(s_{\rm odd \rm}+c_i-c_j+2)}{(c_i+1)(s_{\rm odd \rm}+c_j-c_i)} \nonumber \\
		&= \frac{(c_i + b)(s_{\rm odd \rm} - b+2)}{(c_i+1)(s_{\rm odd \rm} + b)} \nonumber \\
		&= \frac{c_is_{\rm odd \rm}+2c_i-bc_i+bs_{\rm odd \rm}+2b-b^2}{c_is_{\rm odd \rm}+bc_i+b+s_{\rm odd \rm}} \nonumber \\
		&= 1 -(b-1)\frac{2c_i+b-s_{\rm odd \rm}}{c_is_{\rm odd \rm}+bc_i+b+s_{\rm odd \rm}}. \label{codecount2}
	\end{align}
Since there can be at most $c_i+c_j = 2c_i+b$ odd length $\{i,j\}$-segments, we have $2c_i + b \ge s_{\rm odd \rm}$, and $b \ge 2$. The right hand side of \eqref{codecount2} must be less than or equal to $1$ and so
	\begin{align}
		\bP[C_{n,k} \in Q | R_{A,\underline{c}}] \le \bP[C_{n,k} \in Q | R_{A,\underline{c}'}], \nonumber
	\end{align}
as required for \eqref{codecount3}. This completes the proof of the first statement of the lemma. For the second statement we partition $\mathbb{P}[C_{n,k} \in Q]$ over the $P_{\underline{c}}$ to give
\begin{align*}
\bP[C_{n,k} \in Q] &= \sum_{\underline{c}}\bP[C_{n,k} \in Q\cap P_{\underline{c}}]\\
&= \sum_{\underline{c}}\bP[C_{n,k} \in Q|C_{n,k} \in P_{\underline{c}}]\bP[C_{n,k} \in P_{\underline{c}}]\\
&\le \sum_{\underline{c}}\bP[C_{n,k} \in Q|C_{n,k} \in P_{n,k}]\bP[C_{n,k} \in P_{\underline{c}}]\\
&= \bP[C_{n,k} \in Q|C_{n,k} \in P_{n,k}],
\end{align*}
as required.
\end{proof}

We now use Lemma \ref{major} to bound from below the number of Hamilton cycles in $T_k(n)$ and in turn prove Lemma \ref{kKMain}.
\begin{proof}[Proof of Lemma \ref{kKMain}]
Let $k \ge 3$ and suppose $\underline{c} = (c_1,\ldots,c_k) \in \bN^k$ is such that $\sum_{i} c_i = n$. Recalling \eqref{ham4}, we note that if $c_i \le c_j -2$ and we let $\underline{c}'=(c'_1,\ldots,c'_k)$ be such that $c'_i = c_i+1, c'_j = c_j-1$ and $c'_t=c_t$ otherwise, then applying Lemma \ref{major} gives
	\begin{align}
		h(K_{\underline{c}'}) \ge h(K_{\underline{c}}). \label{ham1}
	\end{align}
Furthermore, we have
	\begin{align}
		h(T_k(n)) &= \frac{n!}{2n}\bP[C_{n,k} \in Q | C_{n,k} \in P_{n,k}] \nonumber \\
		&\ge \frac{n!}{2n}\bP[C_{n,k} \in Q] \nonumber \\
		&= \frac{n!}{2n} \bP[a_n \neq a_1,a_{n-1} | a_{n-1} \neq  \cdots \neq a_1]\prod_{i=2}^{n-1}\bP[a_i \neq a_{i-1} | a_{i-1} \neq \cdots \neq a_1] \nonumber \\
		&\ge \frac{n!}{2n}\left(\frac{k-2}{k}\right)\left(\frac{k-1}{k}\right)^{n-2} \nonumber \\
		&= \Omega\left(n^{n-\frac{1}{2}}e^{-n}\left(\frac{k-1}{k}\right)^n\right), \nonumber
	\end{align}
as required.

For $k=2$ we apply a simple counting argument. The number of cycles of length $t =2 \left\lfloor \frac{n}{2} \right\rfloor$ in $T_2(n)$ this is easily counted by ordering both colour classes and accounting for starting vertex and orientation. Therefore we get
 	\begin{align}
		c_t(T_2(n)) = \frac{\left(\left\lfloor \frac{n}{2} \right\rfloor\right)_{\frac{t}{2}} \left(\left\lceil \frac{n}{2} \right\rceil\right)_{\frac{t}{2}}}{2t} = \frac{\left\lfloor \frac{n}{2} \right\rfloor! \left\lceil \frac{n}{2} \right\rceil!}{4\left\lfloor \frac{n}{2} \right\rfloor}, \nonumber
		\end{align}
and the result follows by applying Stirling's approximation. 
\end{proof}
We now use a counting argument to prove Lemma \ref{Turanbest}.
\begin{proof}[Proof of Lemma \ref{Turanbest}]
As before, let $\underline{c} = (c_1,\ldots,c_k) \in \bN^k$ be such that $\sum_{i} c_i = n$. If there exists $i$ and $j$ such that $c_i \le c_j -2$, and we let $\underline{c}'=(c'_1,\ldots, c'_k)$ be such that $c'_i = c_i+1, c'_j = c_j-1$ and $c'_{\ell}=c_{\ell}$ otherwise. We are going to show that $c_r(K_{\underline{c}'}) \ge c_r(K_{\underline{c}})$, for all $r$.

Without loss of generality, we may assume that $i=2$ and $j=1$. We can count the number of cycles of a given length, $r$, by choosing $r$ vertices and then counting the number of Hamilton cycles in graph induced by this cycle and then summing over all choices of $r$ vertices:
  \begin{align}
    c_r(K_{\underline{c}}) = \sum_{\substack{\underline{a} \in \prod_{i=1}^k\{0,\ldots,c_i\} : \\ \sum_{i=1}^k a_i = r}}\biggl[\biggl(\prod_{i=1}^k{c_i \choose a_i}\biggr) \cdot h\bigl(K_{\underline{a}}\bigr)\biggr]. \nonumber
  \end{align}
  
Fix a copy $K$ of $K_{\underline{c}}$ with vertex classes $V_1, \ldots, V_k$ and choose $v \in V_1$; then define $K'$ to be $K \setminus v$ with a vertex $v'$ added to $V_2$ which is a neighbour of all vertices not in $V_2$. We can see that $K'$ is a copy of $K_{\underline{c}'}$. Using this coupling to  compare $c_r(K_{\underline{c}})$ and $c_r(K_{\underline{c}'})$, we only need to consider cycles in $K$ containing $v$ and the cycles in $K'$ containing $v'$. We write $c_{r,v}(G)$ to be the number of cycles of length $r$ in $G$ containing vertex $v$. In what follows we denote the unit vector in direction $m$ by $\underline{e}_m = (y_1,\ldots,y_k)$, where $y_m = 1$ and $y_{\ell} = 0$ otherwise. Since we already assume that $v$ is in our cycle, we then choose $r-1$ other vertices and count the number of Hamilton cycles on the induced subgraph to express $c_{r,v}(K)$ as
	\begin{align}
   		&\sum_{\substack{\underline{a} \in \{0,\ldots,c_1-1\} \times \prod_{i=2}^k\{0,\ldots,c_i\} : \\ \sum_{i=1}^k a_i = r-1}}\biggl[{c_1-1 \choose a_1} \cdot \biggl(\prod_{i=2}^k{c_i \choose a_i}\biggr) \cdot h\bigl(K_{\underline{a}+\underline{e}_1}\bigl)\biggr] \nonumber \\
    		&= \sum_{\substack{a_1 \in \{0,\ldots,c_1-1\} \\ a_2 \in \{0,\ldots,c_2\}}}\biggl[{c_1-1 \choose a_1}{c_2 \choose a_2} \sum_{\substack{(a_3,\ldots,a_k) \in \prod_{i=3}^k\{0,\ldots,c_i\} : \\ \sum_{i=1}^k a_i = r-1}}\biggl[\biggl(\prod_{i=3}^k{c_i \choose a_i}\biggr) \cdot h\bigl(K_{\underline{a}+\underline{e}_1}\bigl)\biggr]\biggr] \nonumber
    	\end{align}
and similarly we may express $c_{r,v'}(K')$ as
	\begin{align}
    		&\sum_{\substack{a_1 \in \{0,\ldots,c_1-1\} \\ a_2 \in \{0,\ldots,c_2\}}}\biggl[{c_1-1 \choose a_1}{c_2 \choose a_2} \sum_{\substack{(a_3,\ldots,a_k) \in \prod_{i=3}^k\{0,\ldots,c_i\} : \\ \sum_{i=1}^k a_i = r-1}}\biggl[\biggl(\prod_{i=3}^k{c_i \choose a_i}\biggr) \cdot h\bigl(K_{\underline{a}+\underline{e}_2}\bigl)\biggr]\biggr] \nonumber \\
		&= \sum_{\substack{a_1 \in \{0,\ldots,c_1-1\} \\ a_2 \in \{0,\ldots,c_2\}}}\biggl[{c_1-1 \choose a_1}{c_2 \choose a_2} \sum_{\substack{(a_3,\ldots,a_k) \in \prod_{i=3}^k\{0,\ldots,c_i\} : \\ \sum_{i=1}^k a_i = r-1}}\biggl[\biggl(\prod_{i=3}^k{c_i \choose a_i}\biggr) \cdot h\bigl(K_{\underline{a}'+\underline{e}_1}\bigl)\biggr]\biggr], \nonumber
  	\end{align}
where $\underline{a}' = (a_2,a_1,a_3,a_4,\ldots,a_k)$ is the vector $\underline{a}$ with the first two values switched.

Define:
	\begin{align}
		\eta(a_1,a_2,\underline{c},r) := \sum_{\substack{(a_3,\ldots,a_n) \in \prod_{i=3}^k\{0,\ldots,c_i\} : \\ \sum_{i=1}^k a_i = r-1}}\biggl[\biggl(\prod_{i=3}^k{c_i \choose a_i}\biggr)h\bigl(K_{\underline{a}+\underline{e}_1}\bigl) \biggr]. \nonumber 
	\end{align}
Then
	\begin{align}
		c_{r,v}(K) = \sum_{\substack{a_1 \in \{0,\ldots,c_1-1\} \\ a_2 \in \{0,\ldots,c_2\}}}{c_1-1 \choose a_1}{c_2 \choose a_2} \eta(a_1,a_2,\underline{c},r) \label{length1}
	\end{align}
and
	\begin{align}
		c_{r,v'}(K') = \sum_{\substack{a_1 \in \{0,\ldots,c_1-1\} \\ a_2 \in \{0,\ldots,c_2\}}}{c_1-1 \choose a_1}{c_2 \choose a_2} \eta(a_2,a_1,\underline{c},r). \label{length2}
	\end{align}
	
If we subtract \eqref{length2} from \eqref{length1} and split the sums depending on the values of $a_1$ and $a_2$, we get
	\begin{align}
		c_{r,v'}(K') - c_{r,v}(K) &=  \sum_{0\le a_2 < a_1 \le c_2} {c_1-1 \choose a_1}{c_2 \choose a_2}\left(\eta(a_2,a_1,\underline{c},r)- \eta(a_1,a_2,\underline{c},r)\right)\nonumber \\
		&+ \sum_{0\le a_1 < a_2 \le c_2} {c_1-1 \choose a_1}{c_2 \choose a_2}\left(\eta(a_2,a_1,\underline{c},r)- \eta(a_1,a_2,\underline{c},r)\right)\nonumber \\
		&+ \sum_{0\le a_2\le c_2<a_1\le c_1-1} {c_1-1 \choose a_1}{c_2 \choose a_2}\left(\eta(a_2,a_1,\underline{c},r)- \eta(a_1,a_2,\underline{c},r)\right). \nonumber
	\end{align}

If we swap around the values of $a_1$ and $a_2$ in the second line of this expression, we get
	\begin{align}
		&c_{r,v'}(K') - c_{r,v}(K) \nonumber \\
		&= \sum_{0 \le a_2<a_1 \le c_2}\biggl({c_1-1 \choose a_1}{c_2 \choose a_2} - {c_1-1 \choose a_2}{c_2 \choose a_1}\biggr)\biggl(\eta(a_2,a_1,\underline{c},r) - \eta(a_1,a_2,\underline{c},r)\biggr) \nonumber \\
		&+ \sum_{\substack{a_1 \in \{c_2+1,\ldots,c_1-1\} \\ a_2 \in \{0,\ldots,c_2\}}}{c_1-1\choose a_1}{c_2\choose a_2}(\eta(a_2,a_1,\underline{c},r)-\eta(a_1,a_2,\underline{c},r)). \label{length4}
	\end{align}

From \eqref{ham1}, we obtain that if $x >y$, then we have $\eta(x,y,\underline{c},r) \le \eta(y,x,\underline{c},r)$. Thus in the first sum of \eqref{length4}, when $a_1 > a_2$, we have $\eta(a_2,a_1,\underline{c},r) - \eta(a_1,a_2,\underline{c},r) \ge 0.$ At the same time, note that since $c_1-1 > c_2$,
	\begin{align}
		{c_1-1 \choose x}{c_2 \choose y} - {c_1-1 \choose y}{c_2 \choose x} > 0 \nonumber
	\end{align}
if and only if $x > y$. Combining these, we must have that for all $0 \le a_2 < a_1 \le c_2$
	\begin{align}
		\biggl({c_1-1 \choose a_1}{c_2 \choose a_2} - {c_1-1 \choose a_2}{c_2 \choose a_1}\biggr)\biggl(\eta(a_2,a_1,\underline{c},r) - \eta(a_1,a_2,\underline{c},r)\biggr) \ge 0 \nonumber
	\end{align}
and so the first sum is positive.

In the second sum of \eqref{length4}, $a_1 > a_2$ and \eqref{ham1} tells us $\eta(a_2,a_1,\underline{c},r) - \eta(a_1,a_2,\underline{c},r) \ge 0.$ Thus the second sum is positive as well. We are then able to conclude that $c_{r,v'}(K') \ge c_{r,v}(K)$ as required.

All that remains is to prove that $c(T_k(n)) > c(G)$ for any $k$-partite graph $G$. Suppose that $G = K_{\underline{c}^0}$ where $\underline{c}^0 = (c^0_1,\ldots,c^0_k) \in \mathbb{N}^k$ is such that $\sum_{i=1}^k c^0_i =n$. While there exist some $i$ and $j$ such that $c^{\ell}_i \le c^{\ell}_j-2$, define $\underline{c}^{\ell+1} = (c^{\ell+1}_1,\ldots,c^{\ell+1}_k)$ by $c^{\ell+1}_i = c^{\ell}_i+1$, $c^{\ell+1}_j = c^{\ell+1}_j-1$ and $c^{\ell+1}_r = c^{\ell}_r$ otherwise. Suppose that this process terminates with $\underline{c}^I$, so $T_k(n) \simeq K_{\underline{c}^I}$. Note that by successive applications of \eqref{ham1}, we have $h(G) \le h(K_{\underline{c}^{I-1}})$. 

We will now show that $h(G) < h(K_{\underline{c}^{I-1}})$. In order to do this, we have to consider \eqref{codecount2} a bit more closely. If $h(G) = h(K_{\underline{c}^{I-1}})$, then at each application of \eqref{ham1}, we have equality. So let us suppose, in order to obtain a contradiction, that $h(K_{\underline{c}^{I-1}}) = h(K_{\underline{c}^{I}})$, for some $I$. In this case, we must have that $s_{\rm odd \rm} = c_i^{I-1}+c_j^{I-1}$ for all $A \in \binom{[n]}{n-(c_i^I+c^I_j)}$. 

Say that a code $a_1,\ldots,a_n$ has an \emph{ij transition} if there exists some $s$ such that $\{a_s,a_{s+1}\} = \{i,j\}$ where indices are taken modulo $n$. For a fixed $A,$ if $s_{\rm odd \rm} = c_i^{I-1}+c_j^{I-1}$ then there can be no $ij$ in any code in $Q$ conditional on $R_{A,\underline{c}^I}$. Therefore if $h(K_{\underline{c}^{I-1}}) = h(K_{\underline{c}^{I}})$ then there are no codes in $Q\cap P_{\underline{c}^I}$ with an $ij$ transition. However we will show that we can construct such a code with an $ij$ transition, and hence obtain our contradiction. We now split into two cases dependent on whether $c^I_i = c^I_j-1$ or $c^I_i = c^I_j$.

First suppose that the $c^I_i = c^I_j-1.$ Since $K_{\underline{c}^I}$ is balanced, all vertex classes are of size $c^I_i$ or $c^I_j.$ In any Hamilton cycle of $K_{\underline{c}^I}$, there must be a transition from a vertex class of size $c^I_i$ to a vertex class of size $c^I_j$ and so by symmetry there must be a Hamilton cycle with a $ij$ transition.

Now suppose that $c^I_i = c^I_j$. If all the vertex class sizes of $K_{\underline{c}^I}$ are the same, then we are done by symmetry. Similarly if the vertex class sizes of $K_{\underline{c}^I}$ are $c^I_i-1$ and $c^I_i,$ then there must be a transition between two classes of size $c^I_i$ and so we are done by symmetry. Finally it remains to consider when $c^I_i = c^I_j$ and the vertex class sizes of $K_{\underline{c}^I}$ are $c^I_i$ and $c^I+1.$ Consider a permutation $\pi = \pi_1\cdots \pi_k$ such that $\pi_{k-1} = i$, $\pi_{k} = j$ and $\{\pi_1,\ldots,\pi_r\} = \{l : c_l^I = c_i^I+1\}$. If $r=1$ and $k=3$, then $c_i^I \ge 2$ (else there are only four vertices) and so the code $\pi_1\pi_2\pi_1\pi_3(\pi_1\pi_2\pi_3)\cdots(\pi_1\pi_2\pi_3)$ is sufficient. If $r=1$ and $k\ge 4$, then the code $\pi_1\pi_2\pi_1\pi_3\pi_4 \cdots \pi_k(\pi_1\cdots \pi_k)\cdots(\pi_1\cdots \pi_k)$ is sufficient. Finally, if $r \ge 2$, then the code $\pi_1 \cdots \pi_r (\pi_1\cdots \pi_k) \cdots (\pi_1\cdots \pi_k)$ is sufficient.

We have shown that there must be an instance of a strict inequality at \eqref{ham1} in the comparison of $h(K_{\underline{c}^{I-1}})$ with $h(K_{\underline{c}^{I}})$. It then follows immediately that $c(T_k(n)) = c(K_{\underline{c}^{I}}) > c(K_{\underline{c}^{I-1}}) \ge c(G)$.
\end{proof}

The proof of Lemma \ref{close} has a similar flavour to that of Lemma \ref{major}. We first prove a preliminary lemma where we evaluate $h_v(2,K_{\underline{c}})$ by considering random codes and then compare $h_v(2,K_{\underline{c}})$ with $h_v(2,K_{\underline{c}'})$. Lemma \ref{close} will follow directly from this next lemma. (For what follows we define $R_{A,\underline{b}}$ as in the proof of Lemma \ref{major}.)

\begin{lem}\label{stepcount}
For $k \ge 3$, suppose $\underline{c}=(c_1,\ldots,c_k) \in \bN^k$ is such that $\sum_{i} c_i = n$ with $0 \neq c_i \le c_j -2$. Let $\underline{c}'=(c'_1,\ldots,c'_k)$ be such that $c'_i = c_i+1, c'_j = c_j-1$ and $c'_{\ell}=c_{\ell}$ otherwise. Suppose $V_1,\ldots, V_k$ and $V_1',\ldots,V_k'$ are the vertex classes of $K_{\underline{c}}$ and $K_{\underline{c}'}$ and pick some $v \in V_1, v' \in V_1'$. Then
	\begin{align}
		h_v(2,K_{\underline{c}}) \le \frac{(c_i+1)c_j}{c_i(c_j-1)}h_{v'}(2,K_{\underline{c}'}). \nonumber
	\end{align}
\end{lem}

\begin{proof}
Recall that $h_v(2,K_{\underline{c}})$ counts orderings $v_1,\ldots,v_n$ of $V(K_{\underline{c}})$ where $v_1 = v$, $v_2 \in V_2$, and $v_1\cdots v_n$ is a Hamilton cycle. There is a bijection between such an ordering and the pair $(C,(\pi_i)_{i\in [k]})$ where: $C$ is a code $a_1\cdots a_n$ on $[k]$ with $a_1 = 1$, $a_2 = 2$ that is in both $Q$ and $P_{\underline{c}}$; and $\pi_i$ is an ordering of $V_i$ for each $i$ and $v$ is the first vertex in $\pi_1$. So if we let $C_{n,k} = a_1\cdots a_n$ be a random code where each $a_i$ is independently and identically uniformly distributed on $[k]$, we have an expression for $h_v(2,K_{\underline{c}}$):
	\begin{align}
		h_v(2,K_{\underline{c}}) = k^n(c_1-1)!\biggl(\prod_{l=2}^k(c_l!)\biggr)\bP[C_{n,k} \in Q \cap P_{\underline{c}}, (a_1,a_2) = (1,2)]. \nonumber
	\end{align}
By considering the multinomial distribution with parameters $n$ and $\left(\tfrac{1}{k},\ldots,\tfrac{1}{k}\right)$ we have
	\begin{align}
		\bP\left[C_{n,k} \in P_{\underline{c}}\right] = \frac{n!}{\prod_{i=1}^k\left(c_i!\right)}k^{-n}, \label{multinomial1}
	\end{align}
and so
	\begin{align}
		h_v(2,K_{\underline{c}}) &= \frac{n!}{c_1}\bP[C_{n,k} \in Q, (a_1,a_2) = (1,2) | C_{n,k} \in P_{\underline{c}}] \nonumber \\
		&= \frac{n!}{c_1}\sum_A\bP[C_{n,k} \in Q, (a_1,a_2) =(1,2) | R_{A,\underline{c}}] \bP[R_{A,\underline{c}} | C_{n,k} \in P_{\underline{c}}]\label{first2}
	\end{align}
where $R_{A,\underline{c}}$ is defined as in the proof of Lemma \ref{major}, and the sum is taken over all $A \in \binom{[n]}{n-(c_i+c_j)}$.

For what follows, we only consider $A \in \binom{[n]}{n-(c_i+c_j)}$ such that $R_{A,\underline{c}} \cap \{(a_1,a_2)=(1,2)\} \neq \emptyset$ as these are the only ones that contribute to \eqref{first2} when considering either $\underline{c}$ and $\underline{c}'$. As in the proof of Lemma \ref{major}, conditioning on $R_{A,\underline{c}}$, let $s_{\rm odd \rm}$ and $s_{\rm even \rm}$ be the number of $\{i,j\}$ subcodes with respectively odd and even lengths, where we consider the code cyclically. Unlike before, we now require $(a_1,a_2) = (1,2)$ and so if one of $i$ and $j$ is $1$ or $2$, one of the subcodes will have a fixed value at $a_1$ and so a fixed starting letter. Let $\chi_{\rm even \rm}$ be the indicator that there is an even length subcode with a fixed first letter. Similarly let $\chi_{\rm odd \rm}$ be the indicator that there is an odd length subcode with a fixed first letter and further let $\chi_{\rm odd \rm}(i)$ and $\chi_{\rm odd \rm}(j)$ be the indicator that there is an odd length subcode with the first letter having fixed value $i$ and $j$ respectively.

As in Lemma \ref{major}, by letting $t = \frac{s_{\rm odd \rm}+c_i-c_j}{2}$ we can now compute $\bP[C_{n,k} \in Q, (a_1,a_2)=(1,2) | R_{A,\underline{c}}]$:
	\begin{align}
		\bP[C_{n,k} \in Q, (a_1,a_2)=(1,2) | R_{A,\underline{c}}] &= \frac{2^{s_{\rm even \rm}-\chi_{\rm even \rm}}{s_{\rm odd \rm}-\chi_{\rm odd \rm} \choose t-\chi_{\rm odd \rm}(i)}}{{c_i+c_j \choose c_i}}, \label{first3} \\
		\bP[C_{n,k} \in Q, (a_1,a_2)=(1,2) | R_{A,\underline{c}'}] &= \frac{2^{s_{\rm even \rm}-\chi_{\rm even \rm}}{s_{\rm odd \rm}-\chi_{\rm odd \rm} \choose t+1-\chi_{\rm odd \rm}(i)}}{{c_i+c_j \choose c_i+1}}. \label{firstz}
	\end{align}

Let $b = c_j-c_i \ge 2$. Note that the $\chi$ values will be the same when considering both $\underline{c}$ and $\underline{c}'$ and so dividing \eqref{first3} by \eqref{firstz} gives
	\begin{align}
        		\frac{\bP[C_{n,k} \in Q, (a_1,a_2)=(1,2) | R_{A,\underline{c}}]}{\bP[C_{n,k} \in Q, (a_1,a_2)=(1,2) | R_{A,\underline{c}'}]} &= \frac{c_j(t+1-\chi_{\rm odd \rm}(i))}{(c_i+1)(s_{\rm odd \rm}-t-\chi_{\rm odd \rm}(j))} \nonumber \\
		&= \frac{c_j}{c_i+1} \cdot \frac{s_{\rm odd \rm}-b+2-2\chi_{\rm odd \rm}(i)}{s_{\rm odd \rm}+b-2\chi_{\rm odd \rm}(j)} \nonumber \\
		&\le \frac{c_j}{c_i+1}\cdot \frac{s_{\rm odd \rm}-b+2}{s_{\rm odd \rm}+b-2}. \label{first4}
	\end{align}
	
Note that $\frac{s_{\rm odd \rm}-b+2}{s_{\rm odd \rm}+b-2}$ is non decreasing in $s_{\rm odd \rm}$ and $s_{\rm odd \rm} \le 2c_i+b = 2c_j-b$, so we can bound \eqref{first4} by taking $s_{\rm odd \rm} = 2c_i+b = 2c_j-b$ to get:
	\begin{align}
		\frac{\bP[C_{n,k} \in Q, (a_1,a_2)=(1,2) | R_{A,\underline{c}}]}{\bP[C_{n,k} \in Q, (a_1,a_2)=(1,2) | R_{A,\underline{c}'}]} &\le \frac{c_j}{c_i+1} \cdot \frac{2c_i + b - b +2}{2c_i + b - b - 2} \nonumber \\
		&= \frac{c_j(c_i+1)}{(c_i+1)(c_j-1)} \nonumber \\
		&= \frac{c_j}{c_j-1}. \label{first5}
	\end{align}
	
If we apply inequality \eqref{first5} to \eqref{first2}:
	\begin{align}
		h_v(2,K_{\underline{c}}) &\le \frac{c_j}{c_j-1}\sum_A\biggl[\frac{n!}{c_1}\bP[C_{n,k} \in Q, (a_1,a_2) = (1,2) | R_{A,\underline{c}'}]\cdot \bP[R_{A,\underline{c}} | C_{n,k} \in P_{\underline{c}}]\biggr]. \nonumber
	\end{align}
Recall that $\bP[R_{A,\underline{c}} | C_{n,k} \in P_{\underline{c}}] = \bP[R_{A,\underline{c}'} | C_{n,k} \in P_{\underline{c}'}]$, so:
	\begin{align}
		h_v(2,K_{\underline{c}}) &\le \frac{c_j}{c_j-1}\sum_A\biggl[\frac{n!}{c_1}\bP[C_{n,k} \in Q, (a_1,a_2) = (1,2) | R_{A,\underline{c}'}]\cdot \bP[R_{A,\underline{c}} | C_{n,k} \in P_{\underline{c}}]\biggr] \nonumber \\
		&= \frac{c_1'c_j}{c_1(c_j-1)}\sum_A\biggl[\frac{n!}{c_1'}\bP[C_{n,k} \in Q, (a_1,a_2) = (1,2) | R_{A,\underline{c}'}]\cdot \bP[R_{A,\underline{c}'} | C_{n,k} \in P_{\underline{c}'}]\biggr] \nonumber \\
		&= \frac{c_1'c_j}{c_1(c_j-1)}h_{v'}(2,K_{\underline{c}'}). \nonumber
	\end{align}
Noting that $\tfrac{c_{\ell}'}{c_{\ell}}$ is maximised by $\ell = i,$ we get
	\begin{align}
		h_v(2,K_{\underline{c}}) \le \frac{(c_i+1)c_j}{c_i(c_j-1)}h_{v'}(2,K_{\underline{c}'}), \nonumber
	\end{align}
	as required.
\end{proof}

We now apply this result to prove Lemma~\ref{close}.

\begin{proof}[Proof of Lemma \ref{close}]
Let $k \ge 3$ and $\underline{c} = (c_1,\ldots,c_n) \in \bN^k$ and suppose $K_{\underline{c}}$ has vertex classes $V_1,\ldots,V_k$. Further suppose $T_k(n)$ has vertex classes $V_1',\ldots, V_k'$ with $b_i=|V_i'|<|V_j'|=b_j$ only if $c_i \le c_j$ and suppose that $v \in V_1 \cap V_1'$. We will prove by induction on $f(c,b) = \sum_i|c_i-b_i|$ that
	\begin{align}
		h_v(2,K_{\underline{c}}) \le h_v(2,T_k(n))\prod_{i =1}^ke^{\left|\log\left(\frac{b_i}{c_i}\right)\right|}. \nonumber
	\end{align}

The base case of $f(c,b) = 0$ follows since $K_{\underline{c}}$ is $T_k(n)$. Suppose that $f(c,b) \ge 1$ and the result holds for smaller values of $f(c,b)$. Note that if $f(c,b) \neq 0$, then since $\sum_i (c_i-b_i) = 0$, there must be $i,j$ such that $c_i \le b_i -1$ and $c_j \ge b_j+1$. Let $i$ and $j$ be such that $b_i-c_i$ and $c_j-b_j$ are maximised. If $b_i =  b_j+1$, we have a contradiction since then $c_i < c_j$, but $b_i > b_j$. This means that $c_j \ge c_i + 2$ and so if we let $\underline{c}'=(c'_1,\ldots,c'_k)$ be such that $c'_i = c_i+1, c'_j = c_j-1$ and $c'_{\ell}=c_{\ell}$ otherwise, we may apply Lemma \ref{stepcount} to get that
	\begin{align}
		h_v(2,K_{\underline{c}}) &\le \frac{(c_i+1)c_j}{c_i(c_j-1)}h_v(2,K_{\underline{c}'}) \nonumber \\
		&= \exp\left\{\left|\log\left(\frac{c_i'}{c_i}\right)\right| + \left|\log\left(\frac{c_j'}{c_j}\right)\right|\right\}h_v(2,K_{\underline{c}'}). \label{close34}
	\end{align}

To proceed by induction, we first observe that $f(c',b) < f(c,b)$ and secondly we must check that if $b_r < b_s$, then $c_r' \le c_s'$. Note that this still holds for $r=i$ and $s=j$ and will still hold if neither $r=i$ nor $s=j$. If $r=i$ and $b_i < b_s$ but $c_i' > c_s'$, then it must be the case that $b_s - c_s > b_i - c_i$, which contradicts our choice of $i$. Similarly if we have $s=j$, $b_r < b_j$ and $c_r' > c_j$, then we arrive at the similar contradiction that $c_r - b_r > c_j - b_j$. Therefore we may apply the inductive hypothesis to \eqref{close34} to conclude that
	\begin{align}
		h_v(2,K_{\underline{c}}) &\le \exp\left\{\left|\log\left(\frac{c_i'}{c_i}\right)\right| + \left|\log\left(\frac{c_j'}{c_j}\right)\right|\right\}h_v(2,T_k(n))\prod_{l =1}^ke^{\left|\log\left(\frac{b_l}{c_l'}\right)\right|} \nonumber \\
		&= h_v(2,T_k(n))\prod_{l \neq i,j}e^{\left|\log\left(\frac{b_l}{c_l}\right)\right|}\prod_{l=i,j}\exp\left\{\left|\log\left(\frac{b_l}{c_l'}\right)\right|+ \left|\log\left(\frac{c_l'}{c_l}\right)\right|\right\} \nonumber \\
		&=  h_v(2,T_k(n))\prod_{i =1}^ke^{\left|\log\left(\frac{b_i}{c_i}\right)\right|}. \nonumber
	\end{align}
\end{proof}

We use a more complicated probabilistic argument for the proof of Lemma \ref{Turancount}. We consider a different version of the random codes we have previously considered.

\begin{proof}[Proof of Lemma \ref{Turancount}]
Let $K$ be a copy of the Tur\'{a}n graph $T_k(n)$ with vertex classes $V_1,\ldots,V_k,$ and fix $b_i = |V_i|$ for each $i \in [k].$ (Note we do not order the sizes of the vertex classes.) Fix $a_1=1$, then given $a_{i-1}$ for $i \ge 2$, let $a_i$ be uniformly distributed on $[k] \setminus \{a_{i-1}\}$. Define the code $C^2(b_1,k) = a_1\cdots a_m$, where $m = \max\{j : |\{i\le j : a_i = 1\}| = b_1\}$ (in other words, keep track of a random walk on $K_{k}$ and stop just before the $(b_1+1)$-th appearance of $1$).

Conditional on $m=n$, the code $C^2(b_1,k)$ is uniformly distributed on codes $f_1\cdots f_n$ in $Q$ that contain $b_1$ copies of $1$ and satisfy $f_1=1$. This is equal in distribution to $C_{n,k} = d_1\cdots d_n$, where each $d_i$ is independently uniformly distributed on $[k]$, conditional on $C_{n,k}$ being in $Q$, having $b_1$ copies of $1$ and starting with $d_1 = 1$. This conditional equivalence between the two random codes allows us to compute bounds in new ways.

Let $W$ be the number of transitions from 1 to 2 in $C^2(b_1,k)$ -- that is $W = |\{j : (a_j,a_{j+1})=(1,2)\}|$. Note that any shift of a code in $Q \cap P_{\underline{b}}$ ($a_{M+1}\cdots a_n a_1 \cdots a_{M}$ for example) will also be in $Q \cap P_{\underline{b}}$. This means that we can shift the code $C^2(b_1,k)$ to each appearance of $1$ to get another instance of a code $f_1 \cdots f_n$ in $Q$, with $f_1=1$ containing $b_1$ appearances of $1$. Thus by symmetry, given $W$, the probability that $C^2(b_1,k)$ starts with $(a_1,a_2)=(1,2)$ is $\frac{W}{b_1}$. We seek to show that $W$ is at most $\frac{b_1}{2k}$ with probability asymptotically smaller than the probability that $C^2(b_1,k)$ is in $P_{\underline{b}}$. With this we know that, conditional on the event $\{C^2(b_1,k) \in P_{\underline{b}}\},$ with high probability $W\ge \tfrac{b_1}{2k}$ and hence by symmetry
	\begin{align*}
		\bP\left[a_2 = 2 | C^2(b_1,k) \in P_{\underline{b}}\right] = \bE\left[\frac{W}{b_1} \bigg| C^2(b_1,k) \in P_{\underline{b}}\right] \ge \frac{1}{2k}(1-o(1)).
	\end{align*}

Since each letter after a copy of $1$ is independently and uniformly distributed on $\{2,\ldots,k\}$ and there are $b_1$ copies of 1, $W$ is distributed like a Binomial random variable $\mathrm{Bin}(b_1,\frac{1}{k-1})$. Applying a Chernoff bounds gives:
	\begin{align}
		\bP\left[W \le \frac{n}{2k^2}\right] \le e^{-\frac{n}{8k^2}}. \label{cher1}
	\end{align}

Now consider the probability that the code $C^2(b_1,k)$ is of the correct length. Note that the letter directly after a $1$ cannot be a $1$ but (until the next copy of $1$), each subsequent letter is a $1$ with probability $\frac{1}{k-1}$ and so removing the letter after each $1$ and considering an appearance of a $1$ as a \emph{failure}, the variable $m-2b_1$ is distributed like a Negative Binomial random variable, $\mathrm{NB}(b_1,\frac{k-2}{k-1})$.
	\begin{align}
		\bP[m=n] &= \bP\left[\mathrm{NB}\left(b_1,\frac{k-2}{k-1}\right) = n-b_1\right] \nonumber \\
		&= {n-(b_1+1) \choose n-2b_1}\left(\frac{k-2}{k-1}\right)^{n-2b_1}\left(\frac{1}{k-1}\right)^{b_1}. \nonumber
	\end{align}
	
Now an application of de Moivre-Laplace (see \cite[VII.3]{Feller-Probability}) tells us that
	\begin{align}
		\bP[m=n] = \Theta\biggl(n^{-\frac{1}{2}}\exp\biggl\{-\frac{(b_1-\frac{n-b_1}{k-1})^2}{2(n-b_1)\frac{k-2}{(k-1)^2}}\biggr\}\biggr). \label{cherz}
	\end{align}

Note that $|b_1 - \frac{n}{k}| < 1$, as the size of a vertex class of a copy of the Tur\'{a}n graph $T_k(n)$ and so $|b_1-\frac{n-b_1}{k-1}| = |\frac{k}{k-1}(b_1-\frac{n}{k})| < 2$. Putting this into \eqref{cherz}, we see that
	\begin{align}
		\bP[m=n] &= \Theta \biggl(n^{-\frac{1}{2}}\exp \biggl\{-O\bigl(n^{-1}\bigr)\biggr\}\biggr) \nonumber \\
		&= \Theta \bigl(n^{-\frac{1}{2}}\bigr). \label{cher2}
	\end{align}
	
Next, consider $\bP\left[C^2(b_1,k) \in P_{\underline{b}} | m = n\right]$. As mentioned above, conditional on $m=n$, $C^2(b_1,k)$ is distributed like $C_{n,k}$ conditional on being in $Q$, starting with $d_1 = 1$ and having $b_1$ copies of $1$. By Lemma \ref{major}, the events $\{C_{n,k} \in P_{\underline{b}}\}$ and $\{C_{n,k} \in Q\}$ are positively correlated and so
	\begin{align}
		\bP[C^2(b_1,k) \in P_{\underline{b}} | m = n] &= \bP[C_{n,k} \in P_{\underline{b}} | C_{n,k} \in Q, d_1 = 1, \text{ $b_1$ copies of $1$}] \nonumber \\
		&\ge \bP[C_{n,k} \in P_{\underline{b}} | C_{n,k} \in Q] \nonumber \\
		&\ge \bP[C_{n,k} \in P_{\underline{b}}]. \nonumber
	\end{align}
Recalling \eqref{multinomial1} and that $\left|b_i-\tfrac{n}{k}\right| < 1$ for all $i,$ Stirling's approximation gives
	\begin{align}
		\bP[C^2(b_1,k) \in P_{\underline{b}} | m = n] &= \Omega(n^{-\frac{k}{2}}). \label{cher3}
	\end{align}

So combining \eqref{cher2} and \eqref{cher3} we can conclude
	\begin{align}
		\bP\left[C^2(b_1,k) \in Q \cap P_{\underline{b}}\right] &= \bP\left[C^2(b_1,k) \in P_{\underline{b}} | m = n\right]\bP[m=n] \nonumber \\
		&= \Omega\left(n^{-\frac{k+1}{2}}\right). \label{cher4}
	\end{align}

We can now complete our proof. We have
	\begin{align}
		h_v\bigl(2,T_k(n)\bigr) &= k^n(b_1-1)!\biggl(\prod_{l=2}^k(b_l!)\biggr)\bP[C_{n,k} \in Q \cap P_{\underline{b}}, (d_1,d_2) = (1,2)] \nonumber \\
		&= k^n(b_1-1)!\biggl(\prod_{l=2}^k(b_l!)\biggr)\bP[C_{n,k} \in Q, d_1 = 1, |\{j : d_j = 1]| = b_1\} \nonumber \\
		&\cdot \bP[C_{n,k} \in P_{\underline{b}}, d_2 = 2 | C_{n,k} \in Q, d_1 = 1, |\{j : d_j = 1\}| = b_1]. \nonumber
	\end{align}

Recall that $C_{n,k} = d_1 \cdots d_n$ given that $C_{n,k} \in Q$ and $d_1 = 1$ and $|\{j : d_j = 1\}| = b_1$ is equal in distribution to $C^2(b_1,k) = a_1 \cdots a_m$ given $m=n$ and so
	\begin{align}
		h_v\bigl(2,T_k(n)\bigr) &= k^n(b_1-1)!\biggl(\prod_{l=2}^k(b_l!)\biggr)\bP[C_{n,k} \in Q, d_1 = 1, |\{j : d_j = 1\}| = b_1] \nonumber \\
		&\cdot \bP[C^2(b_1,k) \in P_{\underline{b}}, a_2 = 2 | m = n] \nonumber \\
		&= k^n(b_1-1)!\biggl(\prod_{l=2}^k(b_l!)\biggr)\bP[C_{n,k} \in Q, d_1 = 1, |\{j : d_j = 1\}| = b_1] \nonumber \\
		&\cdot \bP[a_2 = 2 | C^2(b_1,k) \in P_{\underline{b}}, m=n] \cdot \bP[C^2(b_1,k) \in P_{\underline{b}} | m=n]. \label{cher6}
	\end{align}
	
We can bound $\bP[d_2 = 2 | C^2(b_1,k) \in P_{\underline{b}}, m = n]$ by conditioning on the value of $W$ as follows:
	\begin{align}
		\bP[a_2 = 2 | C^2(b_1,k) \in P_{\underline{b}}, m = n] &\ge \bP\left[a_2 = 2 \bigg| C^2(b_1,k) \in P_{\underline{b}}, m = n, W > \frac{n}{2k^2}\right] \nonumber \\
		&- \bP\left[W \le \frac{n}{2k^2} \bigg| C^2(b_1,k) \in P_{\underline{b}}, m = n \right] \nonumber \\
		&\ge \frac{n}{2k^2b_1} - \frac{\bP[W \le \frac{n}{2k^2}]}{\bP[C^2(b_1,k) \in P_{\underline{b}}, m = n]}. \nonumber
	\end{align}
By applying \eqref{cher1} and \eqref{cher4} we get
	\begin{align}
		\bP[a_2 = 2 | C^2(b_1,k) \in P_{\underline{b}}, m = n] &\ge \frac{n}{2k^2b_1} - O\biggl(\frac{e^{-\frac{n}{8k^2}}}{n^{-\frac{k+1}{2}}}\biggr) \nonumber \\
		&= \frac{n}{2k^2b_1} - o(1). \nonumber
	\end{align}

This means that for sufficiently large $n$, $\bP[a_2 = 2 | C^2(b_1,k) \in P_{\underline{b}}, m = n] \ge \frac{1}{3k}$. Putting this into \eqref{cher6}, we see
	\begin{align}
		h_v\bigl(2,T_k(n)\bigr) &\ge \frac{k^n(b_1-1)!}{3k}\biggl(\prod_{l=2}^k(b_l!)\biggr)\bP[C_{n,k} \in Q, d_1 = 1, |\{j : d_j = 1\}| = b_1] \nonumber \\
		&\cdot \bP[C^2(b_1,k) \in P_{\underline{b}} | m=n] \nonumber \\
		&= \frac{k^n(b_1-1)!}{3k}\biggl(\prod_{l=2}^k(b_l!)\biggr)\bP[C_{n,k} \in Q, d_1 = 1, |\{j : d_j = 1\}| = b_1] \nonumber \\
		&\cdot \bP[C_{n,k} \in P_{\underline{b}} | C_{n,k} \in Q, d_1 = 1, |\{j : d_j = 1\}| = b_1] \nonumber \\
		&= \frac{k^n(b_1-1)!}{3k}\biggl(\prod_{l=2}^k(b_l!)\biggr)\bP[C_{n,k} \in Q \cap P_{\underline{b}}, d_1 = 1] \nonumber \\
		&= \frac{k^n}{2n}\biggl[\prod_{i=1}^k(b_i!)\biggr]\cdot \bP[C_{n,k} \in Q \cap P_{\underline{b}}] \cdot \frac{2n \cdot\bP[d_1 = 1 | C_{n,k} \in Q \cap P_{\underline{b}}]}{3kb_1} \nonumber \\
		&= h\bigl(T_k(n)\bigr) \cdot \frac{2n \cdot\bP[d_1 = 1 | C_{n,k} \in Q \cap P_{\underline{b}}]}{3kb_1}. \nonumber
	\end{align}

By symmetry, $\bP[d_1 = 1 | C_{n,k} \in Q \cap P_{\underline{b}}] = \frac{b_1}{n}$. This completes the proof of the lemma.
\end{proof}

Now we bound below the number of Hamilton cycles in $T_k(n)$ by the number of Hamilton cycles in $T_k(m)$, where $m < n$.

\begin{proof}[Proof of Lemma \ref{recursion}]
Let $v$ be a  vertex contained in the largest vertex class $V_i$ in $T_k(n)$. Removing $v$ gives $T_k(n-1)$. For each Hamilton cycle $v_1\cdots v_{n-1}$ in $T_k(n-1)$, we can form a Hamilton cycle in $T_k(n)$ by inserting $v$ between two vertices $v_j$ and $v_{j+1}$, both not in $V_i$. For each Hamilton cycle in $T_k(n-1)$, there are at least $(n-1)\frac{k-2}{k}$ spaces where we can insert $v$ and under this construction each Hamilton cycle in $T_k(n)$ will be formed in at most one way. Counting over all Hamilton cycles in $T_k(n-1)$, we get that
	\begin{align}
		h(T_k(n)) \ge (n-1)\frac{k-2}{k}h(T_k(n-1)). \label{second1}
	\end{align}

We can apply equation \eqref{second1} inductively to get that for any $i \in [n]$,
	\begin{align}
		h(T_k(n)) \ge (n-1)_i\left(\frac{k-2}{k}\right)^ih(T_k(n-i)). \nonumber
	\end{align}
\end{proof}

We now bound the number of cycles in $T_k(n)$ in terms of the number of Hamilton cycles.
\begin{proof}[Proof of Lemma \ref{secondcount}]
Let $I$ be a subset of $[n]$ with $|I| = r$. Then by Lemma \ref{Turanbest} and Lemma \ref{recursion}, we have
  \begin{align}
    h(G[I]) &\le h(T_k(r)) \nonumber \\
    &\le \left(\frac{k}{k-2}\right)^{n-r}\frac{h(T_k(n))}{(n-1)_{n-r}} \nonumber \\
    &\le \left(\frac{2k}{k-2}\right)^{n-r}\frac{h(T_k(n))}{(n)_{n-r}}. \nonumber
  \end{align}
Summing over all subsets $I$, we have
  \begin{align}
    c(T_k(n)) &\le \sum_{i=0}^{n-3}{\binom{n}{i}}\left(\frac{2k}{k-2}\right)^i\frac{h(T_k(n))}{(n)_i} \nonumber \\
    &= h(T_k(n))\sum_{i=0}^{n-3}\frac{1}{i!}\left(\frac{2k}{k-2}\right)^i \nonumber \\
    &\le e^{\frac{2k}{k-2}}h(T_k(n)), \nonumber
  \end{align}
  as required.
\end{proof}

Finally, we prove Lemma \ref{second2count}.

\begin{proof}[Proof of Lemma \ref{second2count}]
Let $n \in \bN$ and denote $\left\lfloor \frac{n}{2} \right\rfloor$ by $t$ and $\left\lceil \frac{n}{2} \right\rceil$ by $t'$. For $r \ge 2$, the number of cycles of length $2r$ in $T_2(n)$ is $$\frac{\left(t\right)_r\left(t'\right)_r}{2r}.$$ Summing over $r=2,\ldots,t$ gives
	\begin{align}
		c(T_2(n)) &= \sum_{r=2}^t \frac{\left(t\right)_r\left(t'\right)_r}{2r} \nonumber \\
		&= \frac{t! t'!}{2t} \sum_{r=2}^t \frac{t}{r (t-r)! (t'-r)!} \nonumber \\
		&\le \frac{t! t'!}{2t} \sum_{r'=0}^{t-2} \frac{t}{(t-r') r'! r'!}, \nonumber 
	\end{align}
where we substituted $r' = t-r$ to obtain the second equality. As $c_{2t}(T_2(n)) = \frac{t! t'!}{2t}$ and $\frac{t}{(t-s) s!}$ is easily bounded by $2$, we have
	\begin{align}\label{eq:t2}
		c(T_2(n)) &\le 2c_{2t}(T_2(n)) \sum_{r'=0}^{t-2} \frac{1}{r'!} \nonumber \\
		&\le 2c_{2t}(T_2(n)) \sum_{r'\ge0} \frac{1}{r'!} = 2e \cdot c_{2t}(T_2(n)). 
	\end{align}

Let $s = \left\lfloor \frac{n-1}{2} \right\rfloor$ and $s' = \left\lceil \frac{n}{2} \right\rceil$. Note that $t = s'$ and $t' = s+1$, and so
	\begin{align}\label{eq:t22}
	 \frac{n-2}{2}\frac{s'! s!}{2s} \le \frac{s}{t} \cdot\frac{s'! s!t'}{2s}  = \frac{t!t'!}{2t}. 
	\end{align}
Using \eqref{eq:t22} gives
	\begin{align*}
	c_{2\left\lfloor \frac{n-1}{2} \right\rfloor}(T_2(n-1)) =  \frac{s'! s!}{2s} \le  \frac{2}{n-2} \cdot \frac{t!t'!}{2t} = \frac{2}{n-2} c_{2\left\lfloor \frac{n}{2} \right\rfloor}(T_2(n)) . 
	\end{align*}
As $i=o(n)$, repeatedly applying this bound along with \eqref{eq:t2} gives

	\begin{align*}
		c(T_2(n-i)) &\le 2e \cdot c_{2\left\lfloor \frac{n-i}{2} \right\rfloor}(T_2(n-i))\\
		&\le 2e \left(\prod_{j=1}^i \frac{2}{n-j-1}\right) c_{2\left\lfloor \frac{n}{2} \right\rfloor}(T_2(n))\\
		&\le 2e \left(\frac{4}{n}\right)^i c_{2\left\lfloor \frac{n}{2} \right\rfloor}(T_2(n)),
	\end{align*}
	as required.
\end{proof}

\section{Conclusion and Open Questions}\label{sec5}
In this paper we resolve Conjecture~\ref{gundconj} for sufficiently large $n$ (we do not optimise the value of $n$ given by our approach, as it would still be very large). For triangle-free graphs, Arman, Gunderson and Tsaturian~\cite{Gund1} (see also \cite{Gund2}) show that the Tur\'{a}n graph $T_2(n)$ uniquely maximises the number of cycles when $n\ge 141$, but it seems likely that this should hold for all values of $n$.

Theorem~\ref{main} only deals with $H$ such that $\chi(H) \ge 3$ and  $H$ contains a critical edge. When $H$ does not satisfy these properties, our approach is not feasible as the extremal $H$-free graph is no longer $T_k(n)$. It is interesting to consider what could be true for such $H$. For example, it is natural to ask whether it is possible to maximize the number of edges and the number of cycles simulateously (as in Theorem \ref{main}).

\begin{qu}
Let $H$ be a fixed graph. Does $\Ex(n;H)$ contain a graph with $m(n;H)$ cycles for sufficiently large $n$?
\end{qu}

As $T_2(n)$ does not contain any odd cycle, Theorem~\ref{main} implies that for any odd $k$, $T_2(n)$ is the $n$-vertex graph with odd girth at least $k$ containing the most cycles. Arman, Gunderson and Tsaturian~\cite{Gund1} ask a more general question.

\begin{qu}[Arman, Gunderson, Tsaturian~\cite{Gund1}]
What is the maximum number of cycles in an $n$-vertex graph, with girth at least $g$?
\end{qu}

This question seems difficult since comparatively little is known about the maximum number of edges in an graph with girth at least $g \ge 4$.

Another interesting problem was raised by Kir\'{a}ly \cite{deathnote} who asked for the maximum number of cycles in a graph with $m$ edges can contain (without constraining the number of vertices); he conjectured an upper bound of $1.4^m$ cycles. In a recent paper Arman and Tsaturian~\cite{armtsat} give an upper bound of $8.25\times 3^{m/3}$ and a lower bound of $1.37^m$, and conjecture that their upper bound is correct to within a $\left(1+o(1)\right)^m$ factor. It would be interesting to consider the effect of adding the additional constraint of forbidding a subgraph. In particular what is the maximum number of cycles that a triangle-free graph with $m$ edges can contain?

A similar problem to that of Kir\'{a}ly is to maximise the number of cycles in a graph with $n$ vertices and $m$ edges. For $m=\Omega(n^2)$ and $n$ sufficiently large, Arman and Tsaturian \cite[Conjecture 6.1]{armtsat} conjecture a maximum of $\left(1+o(1)\right)^n\left(\tfrac{2m}{en}\right)^n$ cycles. The current best upper bound is $\left(1+o(1)\right)^n\left(\tfrac{2m}{2n}\right)^n$ given in the same paper. We believe that the method used to prove Lemma \ref{easycor} improves this upper bound but does not prove the conjecture. 


Another direction of research is to maximise the number of \emph{induced} cycles. Given a graph $G$, let $m_I(G)$ denote the number of induced cycles in $G$ and let $m_I(n) := \max\{m_I(G): |V(G)| = n\}$. Morrison and Scott~\cite{morsco} recently determined $m_I(n)$ for $n$ sufficiently large and proved that the extremal graphs are unique. The extremal graphs in question are essentially blow-ups of $C_{n/3}$ and contain many copies of $C_4$.

It would be interesting to consider what happens to the extremal graphs when we forbid $C_4$.

\begin{qu}
What is $m_I(n;C_4):= \max\{m_I(G): |V(G)| = n, G \text{ is } C_4\text{-free}\}$?
\end{qu}

\begin{ack}
We would like to thank the referees for their careful reading of the manuscript and detailed comments.
\end{ack}

\bibliography{cycles}
	\bibliographystyle{amsplain}

\end{document}